\documentclass[11pt,reqno]{amsart}
\usepackage{fullpage,enumitem,colonequals}
\usepackage{mathrsfs} 
\usepackage{amsmath,amsfonts,amsthm,amssymb,latexsym,amscd,amsopn,eucal,mathrsfs}

\usepackage[all]{xy} 

\usepackage[dvipsnames]{xcolor}

\usepackage[colorlinks]{hyperref}
\hypersetup{
    colorlinks,
    citecolor=blue,
    linkcolor=red,      
    urlcolor=cyan,
}

\usepackage[OT2,T1]{fontenc}
\DeclareSymbolFont{cyrletters}{OT2}{wncyr}{m}{n}
\DeclareMathSymbol{\Sha}{\mathalpha}{cyrletters}{"58}

\usepackage{color}

\newcommand{\Aff}{\mathbb{A}}
\newcommand{\C}{\mathbb{C}}

\newcommand{\F}{\mathbb{F}}
\newcommand{\G}{\mathbb{G}}

\newcommand{\PP}{\mathbb{P}}
\newcommand{\Q}{\mathbb{Q}}
\newcommand{\R}{\mathbb{R}}
\newcommand{\Z}{\mathbb{Z}}

\newcommand{\calS}{\mathcal{S}}
\newcommand{\calT}{\mathcal{T}}

\newcommand{\Ntot}{N_{\mathrm{tot}}}
\newcommand{\Nloc}{N_{\mathrm{loc}}}

\DeclareMathOperator{\Br}{Br}

\DeclareMathOperator{\disc}{disc}

\DeclareMathOperator{\Gal}{Gal}

\DeclareMathOperator{\im}{im}

\DeclareMathOperator{\Pic}{Pic}

\DeclareMathOperator{\rank}{rank}

\DeclareMathOperator{\Sel}{Sel}

\newcommand{\et}{\text{\'et}}

\newcommand{\els}{{\operatorname{\#ELS}}}
\newcommand{\sol}{{\operatorname{\#soluble}}}

\newcommand{\PGL}{\operatorname{PGL}}

\newcommand{\ra}{\rightarrow}

\newtheorem{theorem}[equation]{Theorem}
\newtheorem{lemma}[equation]{Lemma}

\newtheorem{proposition}[equation]{Proposition}

\theoremstyle{definition}
\newtheorem{definition}[equation]{Definition}

\newtheorem{conjecture}[equation]{Conjecture}

\theoremstyle{remark}
\newtheorem{remark}[equation]{Remark}

\usepackage{microtype}
\usepackage{caption}

\makeatletter
\renewcommand\subsubsection{\@startsection{subsubsection}{2}%
  \z@{.5\linespacing\@plus.7\linespacing}{-.5em}%
  {\normalfont\bfseries}}
\makeatother

\numberwithin{equation}{section}

\usepackage[alphabetic,backrefs,lite,nobysame]{amsrefs} 

\title{Everywhere local solubility for hypersurfaces in products of projective spaces}

\author{Tom~Fisher}
\address{University of Cambridge,
          DPMMS, Centre for Mathematical Sciences,
          Wilberforce Road, Cambridge CB3 0WB, UK}
\email{T.A.Fisher@dpmms.cam.ac.uk}

\author{Wei Ho}
\address{University of Michigan \\ Department of Mathematics \\ Ann Arbor, MI 48109, USA}
\email{weiho@umich.edu}

\author{Jennifer Park}
\address{The Ohio State University \\ Department of Mathematics \\ Columbus, OH 43210, USA}
\email{park.2720@osu.edu}

\date{\today}

\begin{document}

\begin{abstract} 
We prove that a positive proportion of hypersurfaces in products of projective spaces over $\Q$ are everywhere locally soluble, for almost all multidegrees and dimensions, as a generalization of a theorem of Poonen and Voloch \cite{PV}. We also study the specific case of genus $1$ curves in $\PP^1 \times \PP^1$ defined over $\Q$, represented as bidegree $(2,2)$-forms, and show that the proportion of everywhere locally soluble such curves is approximately $87.4\%$. As in the case of plane cubics \cite{BCF-hassecubic}, the proportion of these curves in $\PP^1 \times \PP^1$ soluble over $\Q_p$ is a rational function of $p$ for each finite prime $p$. Finally, we include some experimental data on the Hasse principle for these curves.
\end{abstract}
\maketitle

\section{Introduction}

Let $V$ be a variety defined over $\Q$. The study of rational points on $V$ often involves determining the local points $V(\Q_\nu)$ for completions $\Q_\nu$ of $\Q$. We say that $V$ is {\em (globally) soluble} if the set $V(\Q)$ of rational points is nonempty, and $V$ is \textit{everywhere locally soluble} if the set $V(\Q_\nu)$ is nonempty for all places $\nu \leq \infty$ of $\Q$.

Poonen and Voloch \cite{PV} show that a positive proportion of all hypersurfaces in $\PP^n$ of fixed degree $d$ are everywhere locally soluble for $n, d \geq 2$ and $(n,d) \neq (2,2)$. In particular, they prove that this proportion (as a limit) is exactly the product $c = \prod_\nu c_\nu > 0$ of local factors $c_\nu$, where $c_\nu$ is the proportion of hypersurfaces that have a $\Q_\nu$-point. In \cite{BBL}, Bright, Browning, and Loughran generalized this theorem to other families of varieties. Poonen and Voloch also conjecture that the proportion of globally soluble hypersurfaces is $c$ when $2 \leq d \leq n$ (with $c=0$ when $(n,d) = (2,2)$), and is $0$ when $d > n + 1$. This conjecture implies that the Hasse principle is true for $100\%$ and for $0\%$ of the everywhere locally soluble hypersurfaces in the two cases, respectively. Browning, Le~Boudec, and Sawin \cite{TimPierreWill} have recently proved this conjecture for $2 \leq d \leq n$ and $(n,d) \neq (3,3)$, i.e., that the Hasse principle is satisfied $100\%$ of the time for Fano hypersurfaces of degree $d$ and dimension at least $3$ in $\PP^n$.

However, Poonen and Voloch \cite{PV} do not make any conjectures about local versus global solubility in the boundary (Calabi-Yau) case where $d = n+1$; this case is perhaps the most interesting and the most mysterious. The simplest of these cases, that of plane cubics ($d = 3$ and $n = 2$), has been studied by Bhargava, Cremona, and the first author \cite{BCF-hassecubic}, in which the authors explicitly compute the proportion $c$ of everywhere locally soluble plane cubics; the same authors \cite{BCF-binaryquartic} also study the proportion of everywhere locally soluble hyperelliptic curves. In addition, Bhargava \cite{manjul-hassecubic} shows that a positive proportion of plane cubics fail the Hasse principle and a positive proportion satisfy the Hasse principle.

In this paper, we are interested in answering the analogous questions and computing the explicit proportion of everywhere locally soluble curves defined in \textit{products} of projective spaces. We also study a specific family of genus one curves over $\Q$ (analogous to the plane cubics studied in \cite{BCF-hassecubic}) and determine the proportion of such curves that are everywhere locally soluble. In forthcoming work, Bhargava and the second author show that the Hasse principle fails for a positive proportion of curves in this family \cite{BH-hasse22}. In this paper, we include some data on the success and failure of the Hasse principle for randomly selected sets of these curves with bounded coefficients.

In Section \ref{S: PoonenVoloch}, we begin by proving the analogue of the theorem of Poonen and Voloch for hypersurfaces in products of projective spaces:

\begin{theorem} \label{thm:PVgeneralization}
Let $d_1, \ldots, d_k, n_1, \ldots, n_k$ be positive integers.
The proportion of multidegree $(d_1,\ldots,d_k)$ hypersurfaces over $\Q$ in $\PP^{n_1} \times \cdots \times \PP^{n_k}$ that are everywhere locally soluble tends to a real number $c > 0$, where $c$ is the product $\prod_\nu c_\nu$ over all places $\nu$ and $c_\nu$ is the proportion of such multidegree homogeneous polynomials with a nontrivial zero over $\Q_\nu$, as long as one of the following holds:
\begin{enumerate}
\item $k = 1$ and $n_1 \geq 2$ but not $(n_1,d_1) = (2,2)$;
\item $k \geq 2$.
\end{enumerate}
\end{theorem}
For the cases with $k=1$ not covered by (i), it may be shown that the proportion of everywhere locally soluble hypersurfaces is $0$.

We also predict the success or failure of the Hasse principle for families of hypersurfaces in products of projective spaces in many cases, depending on the multidegree of the hypersurfaces:
\begin{conjecture} \label{conj:PVgeneralization}
For any positive integers $d_i, n_i$ for $i \in \{1,\ldots,k\}$, consider the family of multidegree $(d_1, \ldots, d_k)$-hypersurfaces in $\PP^{n_1} \times \cdots \times \PP^{n_k}$ over $\Q$. Then
\begin{enumerate}
\item If $d_i > n_i+1$ for all $i = 1, \ldots, k$, then $100\%$ of everywhere locally soluble $(d_1, \ldots, d_k)$-hypersurfaces fail the Hasse principle.
\item  If $d_i < n_i+1$ for some $i = 1, \ldots, k$, then $100\%$ of everywhere locally soluble $(d_1, \ldots, d_k)$-hypersurfaces satisfy the Hasse principle.
With these conditions, we have that the proportion of soluble hypersurfaces equals the product of local factors
$c = \prod_\nu c_\nu.$ \label{conj:PVRC}
\end{enumerate}
\end{conjecture}

Note that the numerical conditions in Conjecture \ref{conj:PVgeneralization}(\ref{conj:PVRC}) are significantly more general than just requiring the hypersurface to be Fano, which is equivalent to having $d_i < n_i + 1$ for {\em all} $i$ (not just some $i$).
Still, as in the original Poonen--Voloch heuristic, there are some remaining boundary cases, namely when $d_i \geq n_i + 1$ for all $i$ with equality for at least one $i$. In this paper, we compute the explicit probability of everywhere local solubility for the simplest such  case with $k > 1$, namely the family of bidegree $(2,2)$ curves in $\PP^1 \times \PP^1$. Explicitly, these are the curves defined by the bihomogeneous bidegree $(2,2)$-polynomials of the form
\begin{align} \label{eq:bideg22}
F(X_0, X_1, Y_0, Y_1) = &\  a_{00}X_0^2Y_0^2 + a_{01}X_0^2Y_0Y_1 + a_{02}X_0^2Y_1^2 \nonumber\\ 
&\ + a_{10}X_0X_1Y_0^2 + a_{11}X_0X_1Y_0Y_1 + a_{12}X_0X_1Y_1^2 \\
&\ + a_{20}X_1^2Y_0^2 + a_{21}X_1^2Y_0Y_1 +a_{22} X_1^2Y_1^2 \nonumber
\end{align}
where $a_{ij} \in \Q$ for $i,j \in \{0,1,2\}$. Smooth curves of this form have genus one.
We show:
\begin{theorem}
\label{thm:loc}
For a finite prime $p$, the probability a bidegree $(2,2)$-form with coefficients in $\Z_p$ is soluble over
$\Q_p$ is
\[ \rho(p) = 1 - \frac{p (p-1) (p^2-1) f(p)}{8 (p^8-1) (p^9-1)} \]
where $f(p) = 4 p^{11} - 4 p^{10} + 4 p^9 - p^8 + 5 p^7
   - 2 p^6 + 5 p^5 - p^4 + 2 p^3 - 2 p^2 + 6 p - 2$.
\end{theorem}

The probability a $(2,2)$-form is soluble over the reals (when we
choose the coefficients uniformly at random in $[-B, \ldots, B]$ 
with $B$ large) appears to be about $96.46\%$, based on a Monte Carlo simulation. 
Note that the probability of solubility over $\R$ is greater than $7/8$, since for insolubility
$a_{00},a_{02},a_{20},a_{22}$ must all have the same sign.

Theorem~\ref{thm:loc} gives $\prod_{p < \infty} \rho(p) \approx 0.90592$.
Combining this with the estimate of the previous paragraph, we have $$\prod_{p \le \infty} \rho(p) \approx 0.8739,$$
thereby giving the following theorem:

\begin{theorem}
\label{thm:global}
The probability that a random $(2,2)$-form with coefficients in $\Q$ is $p$-locally soluble for all $p < \infty$ is approximately $90.592\%$; assuming the truth of the Monte Carlo experiment, the probability that a random $(2,2)$-form with coefficients in $\Q$ is everywhere locally soluble is approximately $87.39\%$.
\end{theorem}

A heuristic similar to that in \cite{manjul-hassecubic} suggests that of the
$(2,2)$-forms that are everywhere locally soluble, exactly $1/4$ should be globally soluble. In Section~\ref{sec:exp}, we report on an experiment to test this
prediction numerically. 

\subsection*{Acknowledgments} Much of this work was completed during a workshop organized by Alexander Betts, Tim Dokchitser, Vladimir Dokchitser and Celine Maistret, a trimester organized by David Harari, Emmanuel Peyre, and Alexei Skorobogatov, and a workshop organized by Michael Stoll; we thank the organizers as well as Baskerville Hall, Institut Henri Poincar\'e, and Franken-Akademie Schloss Schney, respectively, for their hospitality during those periods. We also thank Bhargav Bhatt, John Cremona, David Harari, Max Lieblich, Daniel Loughran, Bjorn Poonen, and the anonymous referee for helpful conversations and comments.

WH was partially supported by NSF grants DMS-$1701437$ and DMS-$1844763$ and the Sloan Foundation. JP was partially supported by NSF grant DMS-$1902199$.

\section{Generalization of Poonen--Voloch's theorem and conjecture}
\label{S: PoonenVoloch}

In this section, we generalize the main result and conjectures of \cite{PV} to hypersurfaces in products of projective spaces. The main idea of the proof of the theorem is identical to that of \cite{PV} but relies on a more general combinatorial inequality. We work over the field $\Q$ for this section. 

Fix an integer $k \geq 1$ and positive integers $n_i, d_i$ for $1 \leq i \leq k$. We consider multidegree $(d_1, \cdots, d_k)$ hypersurfaces in the product space $\mathscr{P} := \PP^{n_1} \times \cdots \times \PP^{n_k}$. In particular, let $\Z[\{x_{ij}\}]_{\mathbf d}$ denote the set of multihomogeneous polynomials in $\Z[\{x_{ij}\ : {1 \leq i \leq k, 0 \leq j \leq n_i}\}]$ of multidegree ${\mathbf d} = (d_1,\ldots,d_k)$. There are $m_i := {{n_i + d_i} \choose {d_i}}$ monomials of degree $d_i$ in $(n_i+1)$ variables, so the total number of monomials of multidegree $\mathbf{d}$ in $\Z[\{x_{ij}\}]_{\mathbf d}$ is $m = \prod_{i=1}^k m_i$.

Define the \textit{height} $h(f)$ of $f \in \Z[\{x_{ij}\}]_{\mathbf d}$ to be the maximum of the absolute values of the coefficients of $f$.
Let $M_\Q$ be the set of places of $\Q$. We define the following counts for $H > 0$ and $\nu \in M_{\Q}$:
\begin{align*}
\Ntot(H) &:= \#\{ f \in \Z[\{x_{ij}\}]_{\mathbf d} : h(f) \leq H \} = (2 \lfloor H \rfloor + 1)^m\\
N(H) &:= \#\{ f \in \Z[\{x_{ij}\}]_{\mathbf d} : h(f) \leq H \textrm{ and } \exists\, \mathbf{x} = \{x_{ij}\} \in \prod_{i=1}^k \Z^{n_i+1} \setminus \{0\} \textrm{ with } f(\mathbf{x}) = 0\}\\
N_\nu(H) &:= \#\{ f \in \Z[\{x_{ij}\}]_{\mathbf d} : h(f) \leq H \textrm{ and } \exists\, \mathbf{x} = \{x_{ij}\} \in \prod_{i=1}^k \Q_\nu^{n_i+1} \setminus \{0\} \textrm{ with } f(\mathbf{x}) = 0 \}\\
\Nloc(H) &:= \#\{ f \in \Z[\{x_{ij}\}]_{\mathbf d} : h(f) \leq H \textrm{ and } \forall\, \nu \in M_\Q, \exists\, \mathbf{x} \in \prod_{i=1}^k \Q_\nu^{n_i+1} \setminus \{0\} \textrm{ with } f(\mathbf{x}) = 0 \}.
\end{align*}

In other words, $\Ntot(H)$ is the total number of multidegree $\mathbf{d}$ polynomials in $\mathscr{P}$ of height at most $H$, and $N(H)$, $N_\nu(H)$, and $\Nloc(H)$ are the number of such polynomials that are globally soluble, soluble over $\Q_\nu$, and everywhere locally soluble, respectively. The limits
\[
\lim_{H \to \infty} \frac{N(H)}{\Ntot(H)} \quad \textrm{and} \quad \lim_{H \to \infty} \frac{\Nloc(H)}{\Ntot(H)},
\]
if they exist, will be called the proportion of globally soluble and everywhere locally soluble multidegree $\mathbf{d}$ hypersurfaces, respectively. Note that for any place $\nu$, the local proportion $\lim_{H \to \infty} N_\nu(H)/\Ntot(H)$ exists. Indeed, as explained in Remark 2.3 of \cite{PV}, after normalizing the Haar measure (or Lebesgue measure for $\nu = \infty$) on the space $\Z_\nu^m$ of the multihomogeneous polynomials of degree $\mathbf{d}$ in the variables $\{x_{ij} : 1 \leq i \leq k, 0 \leq j \leq n_i \}$, we see that this local proportion is the measure of the $\nu$-adically closed subset of $\Z_\nu^m$ corresponding to the multihomogeneous polynomials with a nontrivial zero over $\Q_\nu$. We restate Theorem \ref{thm:PVgeneralization} using this language, now ignoring some of the trivial cases where some $d_i = 1$:

\begin{theorem} \label{thm:localproportion}
Let $c_\nu = \lim_{H \to \infty} N_\nu(H)/\Ntot(H)$.
As $H \to \infty$, the proportion $\Nloc(H)/\Ntot(H)$ tends to $c := \prod_\nu c_\nu > 0$,
if one of the following conditions holds:
\begin{enumerate}
\item $k = 1$ and $n_1, d_1 \geq 2$ but not $(n_1,d_1) = (2,2)$;
\item $k = 2$, and if $n_1 = n_2 = 1$, then $d_1, d_2 \geq 2$;
\item $k \geq 3$.
\end{enumerate}
\end{theorem}

\begin{proof}
If $f$ is absolutely irreducible modulo $p$, then for sufficiently large $p$, the Lang-Weil estimate guarantees a smooth point on the hypersurface $f = 0$ modulo $p$, which may be lifted to a $\Q_p$-point by Hensel's lemma. By Lemmas 20 and 21 of \cite{poonenstoll-casselstate}, it suffices to show that the space of reducible polynomials is of codimension at least $2$ in the space of all polynomials in $\Z[\{x_{ij}\}]_{\mathbf d}$. This follows from Lemma \ref{lemma:binom} below, which shows that the product of the projective spaces of multidegree $(r_1,\cdots,r_k)$ polynomials and of multidegree $(d_1-r_1,\cdots, d_k-r_k)$ polynomials has dimension at most $m - 3$, where $m-1$ is the dimension of 
the projective space of multidegree $(d_1,\cdots,d_k)$ polynomials.
\end{proof}

\begin{lemma} \label{lemma:binom}
Fix a positive integer $k$. Given positive integers $n_i, d_i, r_i$ for $1 \leq i \leq k$ such that $0 \leq r_i \leq d_i$ for all $i$ but $0 < \sum_i r_i < \sum_i d_i$, we have
\begin{equation} \label{eq:binominequality}
\prod_{i=1}^k \binom{n_i + r_i}{n_i} + \prod_{i=1}^k \binom{n_i + d_i-r_i}{n_i} < \prod_{i=1}^k \binom{n_i + d_i}{n_i} 
\end{equation}
if one of the following conditions holds:
\begin{enumerate}
\item $k = 1$ and $n_1, d_1 \geq 2$ but not $(n_1,d_1) = (2,2)$; \label{cond:k1}
\item $k = 2$, and if $n_1 = n_2 = 1$, then $d_1, d_2 \geq 2$; \label{cond:k2}
\item $k \geq 3$. \label{cond:k3}
\end{enumerate}
\end{lemma}

\begin{proof}
We define
\[ S = \{ (A_1, \ldots, A_k) : A_i \subset \{1, 2, \ldots, n_i + d_i \} 
\text{ with } |A_i| = n_i \} \]
and subsets
\begin{align*}
S_1 &= \{ (A_1, \ldots, A_k) \in S : A_i \cap \{1, \ldots, r_i\} 
  = \emptyset \text{ for all } 1 \le i \le k \}, \\
S_2 &= \{ (A_1, \ldots, A_k) \in S : A_i \cap \{r_i + 1, \ldots, d_i\} 
  = \emptyset \text{ for all } 1 \le i \le k \}. 
\end{align*}
Then $|S_1 \cap S_2| = 1$. To prove $|S_1| + |S_2| < |S|$ it suffices
to show that $|S \setminus (S_1 \cup S_2)| \ge 2$.
\begin{enumerate}[itemsep=0.5\baselineskip]
\item If $k=1$ then by choosing $A_1$ with
\[ |A_1 \cap \{1, \ldots, r_1\}| = 1 \quad \text{ and } \quad  
|A_1 \cap \{r_1+1, \ldots, d_1\}| = 1 \]
(and so $|A_1 \cap \{ d_1 + 1, \ldots , n_1 + d_1 \}| = n_1 - 2$)
we have $|S \setminus (S_1 \cup S_2)| \ge r_1(d_1 - r_1) \binom{n_1}{2}$,
which is at least 2 under the stated hypotheses.

\item If $k=2$ then by choosing $(A_1,A_2)$ with
\begin{align*}
|A_1 \cap \{1, \ldots, r_1\}| &= \delta, &
|A_1 \cap \{r_1+1, \ldots, d_1\}| &= 1-\delta, \\
|A_2 \cap \{1, \ldots, r_2\}| &= 1-\delta, &
|A_2 \cap \{r_2+1, \ldots, d_2\}| &= \delta, 
\end{align*}
for $\delta = 0,1$, we have 
$|S \setminus (S_1 \cup S_2)| \ge (r_1(d_2 - r_2) + r_2(d_1 - r_1))n_1 n_2$,
which is at least 2 under the stated hypotheses.

\item If $k \ge 3$ then $|S \setminus (S_1 \cup S_2)| \ge \sum_{i \not= j}
r_i(d_j - r_j) \ge 2$. \qedhere

\end{enumerate}
\end{proof}

We also discuss the analogue of the conjecture of Poonen--Voloch in our setting:
\begin{conjecture}  \label{conj:Hassehypersurface}
For any positive integers $d_i, n_i$ for $i \in \{1,\ldots,k\}$, consider the family of multidegree $(d_1, \ldots, d_k)$-hypersurfaces in $\PP^{n_1} \times \cdots \times \PP^{n_k}$ over $\Q$. Then
\begin{enumerate}
\item If $d_i > n_i+1$ for all $i = 1, \ldots, k$, then $100\%$ of everywhere locally soluble $(d_1, \ldots, d_k)$-hypersurfaces fail the Hasse principle.  \label{conj:Hassegeneraltype}
\item If $d_i < n_i+1$ for some $i = 1, \ldots, k$, then $100\%$ of everywhere locally soluble $(d_1, \ldots, d_k)$-hypersurfaces satisfy the Hasse principle.
\label{conj:HasseRC}
\end{enumerate}
\end{conjecture}

Combining Conjecture \ref{conj:Hassehypersurface}(\ref{conj:HasseRC}) with Theorem \ref{thm:localproportion} implies that the proportion of globally soluble hypersurfaces in $\mathscr{P}$ is a product of local densities for many cases (the last part of Conjecture \ref{conj:PVgeneralization}):  If $d_i < n_i+1$ for some $i = 1, \ldots, k$, we have that
\begin{equation} \label{eq:globalequalslocal}
\lim_{H \to \infty} \frac{N(H)}{\Ntot(H)} = c = \prod_\nu c_\nu. 
\end{equation}

\begin{remark}
  If $d_i < n_i + 1$ for some $i$, and yet the hypotheses of Theorem
  \ref{thm:localproportion} are not satisfied then either some
  $d_i = 1$ or $(k,n_1,d_1) = (1,2,2)$. In the first case both sides
  of \eqref{eq:globalequalslocal} are $1$.  In the second case
  concerning plane conics,
  see, e.g., Theorem 2 of \cite{bcfjk}, which states that
  $c_p =1 - \frac{p}{2(p+1)^2}$ for $p$ a prime; and immediately implies
  the result of Serre \cite{serre} 
  that both sides of \eqref{eq:globalequalslocal} are $0$ in this case.
\end{remark}

\subsection{Motivation for Conjecture \ref{conj:Hassehypersurface}(\ref{conj:Hassegeneraltype})}
Fix a point $a = (a_1,\ldots,a_k) \in \Z^{n_1+1} \times \cdots \times \Z^{n_k + 1}$ with each $a_i \neq 0$ and having coprime coordinates, and consider the set of multidegree $(d_1, \ldots, d_k)$-polynomials vanishing on $a$. The set of these polynomials form a hyperplane in $\Z^m$, and to count the multidegree $(d_1, \ldots, d_k)$ polynomials up to height $H$ vanishing on $a$ is to count the set of integral points contained in this hyperplane, whose coordinates are bounded by $H$. As in \cite{PV}, the number of integral points of height at most $H$ in this hyperplane is given by
 \[
 \frac{c(a)H^{m-1}}{\phi(a)} + O(H^{m-2}),
 \]
where $\phi(a)$ denotes the covolume of the lattice of integer points on the hyperplane of polynomials vanishing on $a$, and $c(a)$ denotes the $(m-1)$-dimensional volume of the part of the hyperplane inside $[-1,1]^m$. 

Ignoring the error term, we get that
\[
N(H) \leq H^{m-1} \sum_{a} \frac{c(a)}{\phi(a)},
\]
where the sum ranges over $a \in \Z^{n_1+1} \times \cdots \times \Z^{n_k + 1}$, excluding the zero vectors in each of the components. Since $c(a)$ is bounded by definition, it remains to understand the convergence of $\sum_{a} \frac{1}{\phi(a)}$. Lemma 3.1 of \cite{PV} shows that $\phi(a)$ equals the Euclidean norm of the vector $b$ formed by the monomials of degree $(d_1, \ldots, d_k)$ in the coordinates of $(a_1,\ldots, a_k)$ (the coprime condition in that lemma is not needed). 
Let $\phi_i(a_i)$ denote the Euclidean norm of the vector formed by plugging in the coordinates of $a_i$ into each of the degree $d_i$ monomials in $n_i+1$ variables. Because $\phi(a) = \prod_{i=1}^k \phi_i(a_i)$, we have
\begin{equation} \label{eq:sumphi}
\sum_{a \in (\Z^{n_1+1}\setminus\{0\}) \times \cdots \times (\Z^{n_k + 1}\setminus \{0\})} \frac{1}{\phi(a)}
= \left(\sum_{a_1 \in \Z^{n_1+1}\setminus\{0\}} \frac{1}{\phi_1(a_1)}\right) \cdot \cdots \cdot \left(\sum_{a_k \in \Z^{n_k+1}\setminus\{0\}}\frac{1}{\phi_k(a_k)}\right).
\end{equation}
We claim that each of the sums $\sum_{a_i \in \Z^{n_i+1}\setminus\{0\}}\frac{1}{\phi_i(a_i)}$ converges.

We decompose $\Z^{n_i+1}\setminus\{0\} = T_{n_i+1} \cup T_{n_i} \cup \cdots \cup T_1$, where $T_j$ consists of the vectors in $\Z^{n_i+1}$ with exactly $j$ nonzero coordinates. Let $a_i = (a_{i0},\ldots,a_{in_i})$. Then by applying the AM-GM inequality to the definition of $\phi_i(a_i)$, we have
\begin{equation} \label{eq:amgm}
\phi_i(a_i) \geq m_i^{1/2} \left| \prod_{\ell=0}^{n_i} a_{i\ell} \right|^{\frac{d_i}{n_i+1}}.
\end{equation}
In fact, for $a_i \in T_j$, the product over $\ell$ on the right side of \eqref{eq:amgm} may be taken to be the product of only the $j$ nonzero $a_{i\ell}$ if we replace $m_i$ with $m_{ij} := {j + d_i \choose d_i}$ and $n_i+1$ by $j$.
We thus have
\begin{align*}
\sum_{a \in \Z^{n_i+1}\setminus\{0\}}\frac{1}{\phi_i(a_i)} &= \sum_{j = 1}^{n_i+1}\sum_{a_i \in T_j}\frac{1}{\phi_i(a_i)}\\
&\leq \sum_{j = 1}^{n_i+1}\sum_{a_i \in T_j} m_{ij}^{-1/2} \left| \prod_{\ell \colon a_{i\ell} \neq 0} a_{i\ell} \right|^{-\frac{d_i}{j}}.
\end{align*} 
The above sum converges if $d_i > j$ for all $j = 1, \ldots, n_i +1$, so it converges if $d_i > n_i +1$. Thus, the original quantity \eqref{eq:sumphi} converges if $d_i > n_i + 1$ for all $i = 1, \ldots, k$. 

This heuristic therefore predicts $N(H) = O(H^{m-1})$. Combining this with the estimate $\Ntot(H) \sim (2H)^m$ gives Conjecture \ref{conj:Hassehypersurface}(\ref{conj:Hassegeneraltype}). 

\subsection{Motivation for Conjecture \ref{conj:Hassehypersurface}(\ref{conj:HasseRC})}

We may restrict ourselves to only consider polynomials $f$ for which $f = 0$ defines a smooth geometrically integral hypersurface $X$ in $\mathscr{P}$ (see also \cite[Remark 2.1]{PV}).
We use a well known conjecture of Colliot-Th\'el\`ene (see, e.g., \cite{CT-conj03}):

\begin{conjecture}[Colliot-Th\'el\`ene] \label{conj:CT}
Let $X$ be a smooth proper geometrically integral variety over a number field. If $X$ is (geometrically) rationally connected, then the Brauer-Manin obstruction to the Hasse principle for $X$ is the only obstruction.
\end{conjecture}

In particular, we show in Proposition \ref{prop:RCmultideg} that if $X$ is a multidegree $(d_1,\ldots,d_k)$-hypersurface in $\mathscr{P} = \PP^{n_1} \times \cdots \times \PP^{n_k}$ with $d_i < n_i+1$ for some $i$, then $X$ is rationally connected.
We will show in Proposition \ref{prop:BMvacuous} below that the Brauer-Manin obstruction is vacuous for smooth complete intersections (of dimension at least $3$) in $\mathscr{P}$ over number fields. Thus, conditional on Conjecture \ref{conj:CT}, we see that the Hasse principle is satisfied for these $X$ of dimension $\geq 3$.

In almost all cases with $d_i < n_i + 1$ for some $i$ where $X$ has
dimension $1$ or $2$, either some $d_i = 1$ (and so $X$ has a rational
point) or $X$ is a quadric hypersurface (and so is known to satisfy
the Hasse principle). One nontrivial case is handled in the proof of
\cite{PV}*{Proposition~3.4}, which shows, under
Conjecture~\ref{conj:CT}, that the Hasse principle holds for a density
$1$ set of cubic surfaces in $\PP^2$. The last remaining nontrivial case is that of bidegree $(2,d_2)$-surfaces in $\PP^2 \times \PP^1$
with $d_2 \geq 2$. Here, it is possible for the Hasse principle to
not be satisfied (e.g., the famous example of Iskovskih
\cite{iskovskih-counterex} where $d_2 = 4$), but as we explain in Lemma \ref{lem:conicbundle} below, Conjecture \ref{conj:CT} still implies that, for any fixed $d_2 \geq 2$, the
Hasse principle holds for 100\% of these hypersurfaces. The argument, which relies on the generic member of this family having no Brauer-Manin obstruction by \cite[Theorem 2.6]{iskovskih-rationalsurfaces}, is analogous to that for cubic surfaces.

\begin{proposition} \label{prop:RCmultideg}
Fix an integer $k \geq 1$ and positive integers $d_i, n_i$ for $i \in \{1,\ldots,k\}$. Let $X$ be a smooth integral multidegree $(d_1,\ldots,d_k)$-hypersurface in $\PP^{n_1} \times \cdots \times \PP^{n_k}$ over $\C$. Then $X$ is rationally connected if and only if $d_i < n_i+1$ for some $i$.
\end{proposition}

\begin{proof}
One direction is easy: if $d_i > n_i$ for all $i$, then $H^0(X,\omega_X)$ is nonzero, hence $X$ is not rationally connected.
For the reverse direction, without loss of generality, suppose $d_1 < n_1 + 1$. Then consider the projection $\pi: X \to \PP^{n_2} \times \cdots \times \PP^{n_k}$ (note that this codomain is a point if $k =1$). The fibers of this projection $\pi$ are degree $d_1$ hypersurfaces in $\PP^{n_1}$; the general ones are smooth and, by assumption, Fano and thus rationally connected \cite{campana,kollarmiyaokamori}. 

We apply a theorem of Graber--Harris--Starr \cite{graberharrisstarr}, which states that if $X \to Y$ is a dominant morphism of complex varieties where the general fiber and $Y$ are both rationally connected, then $X$ is rationally connected as well.
Here the projection map $\pi$ is surjective by a dimension count, and the codomain $\PP^{n_2} \times \cdots \times \PP^{n_k}$ is rationally connected, so $X$ itself is rationally connected.
\end{proof}

In the following proposition, the notation $\Br Y$ refers to the cohomological Brauer group $H^2_\et(Y,\G_m)$ for a scheme $Y$.

\begin{proposition} \label{prop:BMvacuous}
Let $K$ be a number field. Let $X$ be a smooth complete intersection in a product $\mathscr{P} =  \PP^{n_1} \times \cdots \times \PP^{n_k}$ satisfying $\dim X \geq 3$. Then the natural map $\Br K \to \Br X$ is an isomorphism, hence the Brauer-Manin obstruction for $X$ is vacuous.
\end{proposition}
\begin{proof}
We need only slightly modify the proof in \cite[Appendix A and B]{PV} for the same result where $X$ is a smooth complete intersection in a single projective space $\PP^n$ (see also \cite[Proposition 2.6]{Schindler} for a smooth complete intersection in the product of two projective spaces). We summarize the argument here. We use the low degree exact sequence from the Leray spectral sequence
\begin{equation} \label{eq:lowdegexactseq}
0 \to \Pic X \to (\Pic X_{\overline{K}})^{G_K} \to \Br K \to \ker (\Br X \to \Br  X_{\overline{K}}) \to H^1(K, \Pic  X_{\overline{K}})
\end{equation}
where $G_K = \Gal(\overline{K}/K)$. We need to check that $ \Pic X \to (\Pic X_{\overline{K}})^{G_K}$ is an isomorphism, $H^1(K, \Pic  X_{\overline{K}}) = 0$, and $\Br X_{\overline{K}} = 0$.

The restriction map $\Pic \mathscr{P}_{\overline{K}} \to \Pic X_{\overline{K}}$ is an isomorphism since $X$ is a smooth complete intersections of dimension at least $3$ \cite{hartshorne-amplebook}*{Corollary 3.3}. Then since $\Pic \mathscr{P} \to \Pic \mathscr{P}_{\overline{K}}$ is an isomorphism, we find that the injections
\[ \Pic X \hookrightarrow (\Pic X_{\overline{K}})^{G_K} \hookrightarrow \Pic X_{\overline{K}} \]
are isomorphisms and $H^1(K, \Pic X_{\overline{K}}) = 0$.

To show that $\Br X_{\overline{K}} = 0$ here, we use the Kummer sequence to obtain, for any prime $\ell$,
\[
\xymatrix{
0 \ar[r] & (\Pic \mathscr{P}_{\overline{K}})/ \ell \ar[r]  \ar[d] & H^2(\mathscr{P}_{\overline{K}},\Z/\ell\Z) \ar^{\psi}[d] \\
0 \ar[r] & (\Pic X_{\overline{K}})/ \ell \ar[r] & H^2 (X_{\overline{K}},\Z/\ell\Z) \ar[r] & (\Br X_{\overline{K}})[\ell] \ar[r] & 0.
}
\]
The top horizontal injection is in fact an isomorphism, since both groups are rank $k$ over $\Z/\ell\Z$ (using the Kunneth formula to compute $H^2(\mathscr{P},\Z/\ell\Z)$). By a version of Weak Lefschetz (see, e.g., \cite{PV}*{Corollary B.5} with $V = \mathscr{P}$), the vertical map $\psi$ is an isomorphism, so we have $(\Br X_{\overline{K}})[\ell] = 0$ for all $\ell$. Since $\Br X_{\overline{K}}$ is torsion, it is in fact $0$.

Thus, from \eqref{eq:lowdegexactseq}, the map $\Br K \to \Br X$ is an isomorphism, so no elements of $\Br X$ obstruct rational points on $X$.
\end{proof}

\begin{lemma} \label{lem:conicbundle}
Assume Conjecture \ref{conj:CT}. Fix an integer $d \geq 2$ and consider the family of bidegree $(2,d)$ hypersurfaces in $\PP^2 \times \PP^1$. Then $100\%$ of everywhere locally soluble such hypersurfaces satisfy the Hasse principle.
\end{lemma}
\begin{proof}
These bidegree $(2,d)$ surfaces $S$ are conic bundles over $\PP^1$, say given by equations of the form $f(x_1,x_2,x_3,y_1,y_2) = 0$, where $x_1, x_2, x_3$ and $y_1, y_2$ are coordinates for $\PP^2$ and $\PP^1$, respectively. We may also represent $f$ as a symmetric $3 \times 3$ matrix $M(y_1,y_2)$ of degree $d$ forms in $y_1, y_2$, i.e., such that $f(x,y) = \vec{x}^{\text{T}}M \vec{x}$. Let $Z$ be the singular locus in $\PP^1$ of the fibration $S \to \PP^1$; explicitly, $Z$ is defined by the determinant of the matrix $M(y_1,y_2)$, which has degree $3d$ in $y_1, y_2$, and degree $3$ in the coefficients $a_1,\ldots,a_{6(d+1)}$ of $f$.

For the generic member of this family over $\Q(a_1,\ldots,a_{6(d+1)})$, we claim that $Z$ is irreducible. This follows, e.g., from specialization, since for a prime $p$, the determinant of the $3 \times 3$ matrix
$$\begin{pmatrix} 0 & y_2^d & y_1^d \\ y_2^d & y_1^d & 0 \\ y_1^d & 0 & p y_2^d \end{pmatrix}$$
is irreducible by Eisenstein's criterion. Thus, the singular locus is irreducible for the generic surface over $\Q(a_1,\ldots,a_{6(d+1)})$, and by Hilbert irreducibility, the same is true for a density $1$ set of such surfaces over $\Q$. Finally, if $Z$ is irreducible, Theorem 2.6 of \cite{iskovskih-rationalsurfaces} immediately implies that there is no Brauer-Manin obstruction for the surface $S$, so by Conjecture \ref{conj:CT}, the Hasse principle holds.
\end{proof}

\section{Counting polynomials over finite fields}
\label{sec:count}

Our strategy for proving Theorems \ref{thm:loc} and \ref{thm:global} can be viewed as an extension of Hensel's lemma. If the reduction of the $(2,2)$-form mod $p$ has a smooth $\F_p$-point, then Hensel's lemma implies that this can be lifted to a $\Q_p$-point; conversely, if the reduction of the $(2,2)$-form has no $\F_p$-points, then clearly it has no $\Q_p$-points. The $p$-adic solubility of many $(2,2)$-forms can be determined in this way, so it is crucial to understand $(2,2)$-forms over finite fields. We work over a finite field $\F_q$, where $q$ is a power of a prime $p$, for this section, since it is no extra work to do so, though we will only need these results over $\F_p$ for the later sections.

\subsection{Preliminaries}

Many $(2,2)$-forms modulo $p$ contain binary quadratic forms as one of its factors, so we collect some relevant calculations here. These are easy to check.

\begin{lemma}\label{L: affinebinquad}
Of the $q^2$ monic quadratic polynomials $f \in \F_q[X]$,
\begin{itemize}
\item $q(q-1)/2$ have two distinct roots in $\F_q$;
\item $(q^2-q)/2$ have conjugate roots in $\F_{q^2}$;
\item $q$ have a double root in $\F_q$.
\end{itemize}
\end{lemma}

\begin{lemma}
\label{lem:countbq}
Of the $q^3$ binary quadratic forms $f \in \F_q[X,Y]$,
\begin{itemize}
\item $(q-1)(q+1)q/2$ have two distinct roots in $\F_q$;
\item $(q-1)(q^2-q)/2$ have two conjugate roots in $\F_{q^2}$;
\item $(q-1)(q+1)$ have a double root in $\F_q$;
\item $1$ is the zero form.
\end{itemize}
\end{lemma}

\subsection{Reducible $(2,2)$-forms}

Now we look at the bihomogeneous polynomials of bidegree $(d_1,d_2)$ in $\F_q[X_0, X_1; Y_0,Y_1]$ for $0 \leq d_0, d_1 \leq 2$, starting with the irreducible $(d_1, d_2)$-forms with $(d_1, d_2) \neq (2,2)$. These correspond to $(d_1,d_2)$-curves in $\PP^1 \times \PP^1$, where we assume that $(X_0, X_1)$ corresponds to the coordinates in the first factor of $\PP^1$ and $(Y_0, Y_1)$ corresponds to the coordinates in the second $\PP^1$.

\begin{lemma}
\label{count:irred}
The number of irreducible bihomogeneous polynomials in $\F_q[X_0,X_1;Y_0,Y_1]$ 
(with bidegrees as indicated) 
are as follows:

\[ \begin{array}{c|lll}
\textup{Bidegree} & 
\multicolumn{3}{c}{\textup{Number of forms up to scaling by $\F_q^\times$}}\\
\hline
(1,0) 
& \hspace{3em} & m_{10} = q+1  \\ 
(2,0) 
&& m_{20} = (q^2 - q)/2 \\
(1,1) && m_{11} = q^3-q  \\  
(2,1) 
&& m_{21} = q^5-q^3    
\end{array} \]
\end{lemma}

\begin{proof}
  There are $q^2-1$ nonzero $(1,0)$-forms, and so $q+1$ up to scaling.
  The number of irreducible $(2,0)$-forms is taken from
  Lemma~\ref{lem:countbq}, taking into account scaling. The
  coefficients of a $(1,1)$-form may naturally be arranged as a 2 by 2
  matrix, and the form is irreducible if and only if this matrix is
  nonsingular. Therefore each irreducible $(1,1)$-form defines the graph of
  a Mobius map, and so $m_{11} = |\PGL_2(\F_q)| = q^3 - q$.  Finally
  we compute
  \[ m_{21} = (q^6-1)/(q-1) - m_{10}m_{11} - m_{10} m_{20} -
  m_{10}^2(m_{10}+1)/2 = q^5-q^3. \qedhere \]
\end{proof}

We now consider all the different ways in which a $(2,2)$-form can factor.
We use the notation $(a_1,b_1)^{e_1}\cdots (a_r,b_r)^{e_r}$ to denote the bidegrees of the irreducible factors, with multiplicity.
For example, the {\em factorization type} $(1,0)^2(0,1)(0,1)$ indicates that 
the $(2,2)$-form factors as a product $F_{10}^2 F_{01} G_{01}$, where 
$F_{10}$, $F_{01}$, and $G_{01}$ are irreducible polynomials in $\F_q[X_0,X_1;Y_0,Y_1]$ with 
bidegrees $(1,0)$, $(0,1)$, and $(0,1)$, respectively, and that $F_{01}$ 
is not an $\F_q^\times$-multiple of $G_{01}$. 

\begin{lemma} \label{lemma:number for each splitting type}
The number of reducible $(2,2)$-forms over $\F_q$ with each factorization
type are as follows. Moreover, the curve $C \subset \PP^1 \times \PP^1$ 
defined by such a form either always has a smooth $\F_q$-point,
or never has a smooth $\F_q$-point, as indicated in the right hand column.

\[ \hspace{-0.5em} \begin{array}{cr@{\,\,}c@{\,\,}lc}
\text{\rm{Factorization type}} & 
\multicolumn{3}{c}{\hspace{-1em}\text{\rm{Number of forms up to scaling by $\F_q^\times$}}} 
& \text{{\rm Smooth point?}} \\
\hline
(1,1)(1,1)               & \binom{m_{11}}{2} &=&  (q^3 - q) (q^3 - q - 1)/2 & \rm{yes} \\    
(2,1)(0,1) \text{ {\rm or} } (1,2)(1,0)  & 2 m_{21} m_{10} &=&  2 q^3 (q + 1)^2 (q - 1) & \rm{yes} \\    
(1,1)(1,0)(0,1)          & m_{11} m_{10}^2 &=& q (q + 1)^3 (q - 1) &               \rm{yes} \\    
(1,0)(1,0)(0,1)(0,1)    & \binom{m_{10}}{2}^2 &=&  q^2 (q + 1)^2/4 &                   \rm{yes} \\    
(2,0)(0,1)(0,1) \text{ {\rm or} } (0,2)(1,0)(1,0)  & 
                          2 m_{20} \binom{m_{10}}{2} &=& q^2 (q + 1) (q - 1)/2 &  \rm{yes} \\    
(2,0)(0,2)               & m_{20}^2 &=& q^2 (q - 1)^2/4 &                   \rm{no} \\    
(1,0)^2(0,1)(0,1) \text{ {\rm or} } (0,1)^2(1,0)(1,0)
                      & 2 m_{10} \binom{m_{10}}{2} &=& q (q + 1)^2 &                     \rm{yes} \\    
(2,0)(0,1)^2 \text{ {\rm or} } (0,2)(1,0)^2 & 2m_{20} m_{10} &=& q (q + 1) (q - 1) &    \rm{no} \\  
(1,1)^2                  & m_{11} &=& q (q + 1) (q - 1) &                  \rm{no} \\
(1,0)^2(0,1)^2           & m_{10}^2 &=& (q + 1)^2 &                         \rm{no}  
\end{array}  \hspace{-0.5em} \]
\end{lemma}

\begin{proof}
The counts all follow from Lemma~\ref{count:irred} as indicated.
In each of the cases listed as having smooth $\F_q$-points, there is 
an irreducible factor of multiplicity 1 with bidegree $(1,0)$, $(0,1)$, $(1,1)$, $(2,1)$ or $(1,2)$.
Each such factor defines a smooth curve of genus $0$ which, by projection to 
one of the factors, is isomorphic to $\PP^1$. Moreover a case-by-case 
analysis shows 
that this curve meets the other components of $C$ in at most 2 points. 
Since $\# \PP^1(\F_q) = q + 1 > 2$ this shows
that $C$ has a smooth $\F_q$-point. In the remaining cases each irreducible factor
is either repeated or has bidegree $(2,0)$ or $(0,2)$. So in these cases
there are no smooth $\F_q$-points.
\end{proof}

\subsection{Irreducible $(2,2)$-forms}

We now consider the irreducible $(2,2)$-forms over $\F_q$. 
We distinguish between those that are absolutely irreducible (i.e., do not factor
over $\overline{\F}_q$) and those that factor over $\F_{q^2}$ as the product of a bidegree $(1,1)$-form 
and its conjugate.
In the latter case, we say the form has
factorization type $(1,1)\overline{(1,1)}$.

\begin{lemma}
\label{lem:irred}
Let $C \subset \PP^1 \times \PP^1$ be a curve defined by an absolutely 
irreducible $(2,2)$-form $F \in \F_q[X_0,X_1;Y_0,Y_1]$. Then $C$ has a smooth $\F_q$-point.
\end{lemma}

\begin{proof}
If $C$ is smooth then it has genus 1, and the lemma
follows by the Hasse-Weil bound. If $C$ is singular, then it has geometric genus $0$. The normalization is a smooth genus $0$ curve, in fact $\PP^1$ itself (e.g., by the Hasse-Weil bound), and thus has $(q+1)$ $\F_q$-points. Since the preimage of the singular point is at most length $2$, the curve $C$ must have at least one smooth $\F_q$-point.
\end{proof}

\begin{remark}
\label{rem:morecounts}
The numbers of irreducible $(2,2)$-forms are as follows:
 \[ \begin{array}{clll}
\text{\rm{Factorization type}} & 
\multicolumn{3}{c}{\text{\rm{Number of forms up to scaling by $\F_q^\times$}}} \\
\hline
\text{smooth}             & \hspace{2.5em} &q^4 (q + 1)^2 (q - 1)^2 \\
\text{absolutely irreducible yet singular}        & &q^3 (q + 1)^2 (q - 1)^2 \\
(1,1)\overline{(1,1)}    & &(q^3 - q)(q^3 + q - 1)/2 
\end{array} \]
In the case $(1,1)\overline{(1,1)}$ the $(1,1)$-forms are 
irreducible over $\F_{q^2}$, since the original $(2,2)$-form was 
irreducible over $\F_q$.
Each therefore defines the graph of a Mobius map, and so the count in 
the last row is $(\#\PGL_2(\F_{q^2}) -\#\PGL_2(\F_{q}))/2$.
We omit the details of the other two counts since these are not
needed for the proof of Theorem~\ref{thm:loc}. 
However, as a check on our calculations, we note that these counts, together 
with those in Lemma~\ref{lemma:number for each splitting type}, 
do indeed add up to $(q^9-1)/(q-1)$. 
\end{remark}

If $F$ has factorization type $(1,1)\overline{(1,1)}$ then $C$ is
geometrically the union of two rational curves. We subdivide into
cases according as these meet in
\begin{enumerate}
\item a pair of points defined over $\F_q$, 
\item a pair of conjugate points defined over $\F_{q^2}$, 
\item a single point defined over $\F_q$.
\end{enumerate}

\begin{lemma} 
\label{lem:count1}
The number of $(2,2)$-forms over $\F_q$ (up to scaling by $\F_q^\times$) in cases (i), (ii) and (iii) above are, respectively,
\begin{align*}
 n_{11} &= q^3 (q + 1)^2 (q - 1)/4, \\
 n_{12} &= q^2 (q + 1) (q - 1)^2 (q - 2)/4, \\
 n_{13} &= q (q + 1)^2 (q - 1)^2/2.
\end{align*}
\end{lemma}
\begin{proof}  \ 
\begin{enumerate}[leftmargin={2\parindent},itemsep=0.5\baselineskip]
\item \label{lem:count1-1}
The singular points can be any pair of points in $\PP^1(\F_q) \times 
\PP^1(\F_q)$ that remain distinct under both projection maps. 
This last condition comes from the fact, noted in Remark~\ref{rem:morecounts},
that each $(1,1)$-form defines the graph of a Mobius map.
We make a change of coordinates to move the singular points to
$((0:1),(0:1))$ and $((1:0),(1:0))$. By hypothesis $F$
factors as the product of two $(1,1)$-forms over $\F_{q^2}$, and these are now
linear combinations of $X_0Y_1$ and $X_1Y_0$. Therefore
$F = f(X_0Y_1,X_1Y_0)$ for some irreducible binary
quadratic form $f$. We compute $n_{11}$ as the product of the
$(q^2 - q)/2$ choices for $f$ (up to scaling), and the $q^2(q+1)^2/2$ 
choices for the (unordered) pair of singular points.
\item 
\label{lem:count1-2}
We write $\F_{q^2} = \F_{q}(\alpha)$ where $\alpha^2 + r \alpha + s = 0$ 
for some $r,s \in \F_q$. As noted in (i), the singular points remain
distinct under both projection maps. After a change of coordinates, 
defined over $\F_q$, we may therefore assume that the singular points 
are $((\alpha:1),(\alpha:1))$ and its Galois conjugate.
Then $F = f(X_0 Y_0 + r X_0 Y_1 + s X_1 Y_1,X_0 Y_1 - X_1 Y_0)$
for some irreducible binary quadratic form $f$. There are $(q^2 - q)/2$ choices 
for $f$ (up to scaling), but one of these gives $F = (X_0^2 + r X_0 X_1 + s X_1^2)
(Y_0^2 + r Y_0 Y_1 + s Y_1^2)$. Therefore $n_{12}$ is the product of $(q+1)(q-2)/2$
and the $(q^2 - q)^2/2$ choices for the singular points.
\item \label{lem:count1-3}
We make a change of coordinates to move the singular point to
$((0:1),(0:1))$. Then $F$ is the product of 
$\alpha X_0 Y_0 + \beta X_0 Y_1 + \gamma X_1 Y_0$,
and its Galois conjugate, for some $\alpha,\beta,\gamma \in \F_{q^2}$
with $\beta,\gamma \not= 0$. 
If $\beta$ and $\gamma$ are a basis for
$\F_{q^2}$ as an $\F_q$-vector space, then by substitutions of the
form $X_1 \leftarrow X_1 + \lambda X_0$ and 
$Y_1 \leftarrow Y_1 + \mu Y_0$ with $\lambda,\mu \in \F_q$ we may
reduce to the case $\alpha =0$. But then the two components also
meet at $((1:0),(1:0))$, which cannot happen in case (iii).
Therefore
$F = f(X_0Y_0,X_0Y_1 + c X_1Y_0)$ for some irreducible binary
quadratic form $f$ and constant $c \in \F_q^\times$. 
We compute $n_{13}$ as the product of the $(q^2 - q)/2$ choices for 
$f$ (up to scaling), the $q-1$ choices for $c$, and the $(q+1)^2$ choices
for the singular point.
 \end{enumerate}

As a final check, 
we note that $n_{11} + n_{12} + n_{13} = (q^3 - q)(q^3 + q - 1)/2$.
\end{proof}

\section{Local solubility for bidegree $(2,2)$-forms} \label{sec:mainpf}

Fix a prime $p$, and consider the space of all $(2,2)$-forms $F \in \Z_p[X_0,X_1;Y_0,Y_1]$ with its natural product Haar measure when viewed as a copy of $\Z_p^9$. In this section, we determine the density of $\Q_p$-soluble forms in this space.

In order to determine the solubility of a given form $F$, it will
often suffice to look at its reduction mod $p$, denoted
$\overline{F} \in \F_p[X_0,X_1;Y_0,Y_1]$, and look for a smooth
$\F_p$-point on the curve $\overline{C}$ defined by $\overline{F}$, so
that we may apply Hensel's lemma. As seen in the tables of
Section~\ref{sec:count} (where we now take $q=p$), it is easy to see
that most of the factorization types for $\overline{F}$ have a smooth
$\F_p$-point on $\overline{C}$.  According to the results of
Section~\ref{sec:count}, only five cases require further
consideration, which we analyze in Section~\ref{ss: nosmoothpts}.

\subsection{Preliminaries}

Let $v(a) = v_p(a)$ denote the $p$-adic valuation of $a \in \Q_p$.

We will often keep track of the valuations of the coefficients of the $(2,2)$-form \eqref{eq:bideg22} as a $3 \times 3$ table:
\begin{equation} \label{eq:valtable}
\begin{matrix}
v(a_{00}) & v(a_{01}) & v(a_{02})\\
v(a_{10}) & v(a_{11}) & v(a_{12})\\
v(a_{20}) & v(a_{21}) & v(a_{22}){\rlap ,}
\end{matrix}
\end{equation}
where each entry is the valuation of the coefficient of the corresponding monomial term in
\begin{equation} \label{eq:coefftable}
\begin{matrix}
X_0^2Y_0^2 & X_0^2 Y_0Y_1 & X_0^2Y_1^2\\
X_0X_1Y_0^2 & X_0X_1Y_0Y_1 & X_0X_1Y_1^2\\
X_1^2Y_0^2 & X_1^2Y_0Y_1 & X_1^2Y_1^2{\rlap .}
\end{matrix}
\end{equation}

\subsection{A useful lemma} Before analyzing the cases where Hensel's lemma does not directly apply, we note that the following lemma will be used several times in the next section, 
and is, in some sense, typical of the arguments we use. In fact, we use it once in the analysis of
Case 3 and twice in the analysis of Case 5 (in Lemmas~\ref{lem:nu51} and~\ref{lem:nu5prime}).

\begin{lemma}
\label{L: solubility1/2}
Let $a_{00}, a_{01}, a_{02}, a_{10}, a_{11}, a_{20}$ be any fixed elements of $\Z_p$ satisfying
\begin{align*}
&v(a_{00}) \geq 2, \quad v(a_{01}) \geq 2, \quad v(a_{02})=1,\\
&v(a_{10}) \geq 1, \quad v(a_{11}) \geq 1,\\
&v(a_{20}) = 0.
\end{align*}
Let 
\[
\calS := \left\{\sum_{i,j = 0}^2a_{ij}X_0^{2-i}X_1^iY_0^{2-j}Y_1^j : 
  a_{12} \in p\Z_p, a_{21} \in \Z_p, a_{22} \in \Z_p\right\}. \]
Then the proportion of the polynomials in $\calS$ that have $\Q_p$-solutions for which $p \nmid X_1Y_1$  is $1/2$.
\end{lemma}
\begin{proof}
Let $\sigma(a_{00},a_{01},a_{02},a_{10},a_{11},a_{20})$ be the desired probability (that a polynomial in $\calS$ has a $\Q_p$-solution with $p \nmid X_1 Y_1$), and let
$\tau(a_{00},a_{01},a_{02},a_{10},a_{11},a_{20})$ be the corresponding probability when 
$\calS$ is replaced by its subset
\[
\calT := \left\{\sum_{i,j = 0}^2a_{ij}X_0^{2-i}X_1^iY_0^{2-j}Y_1^j : 
  a_{12}, a_{21}, a_{22} \in p\Z_p\right\}. \]

If $F \in \calS$, then $F(0,1;Y_0,1)$ reduces mod $p$ to a quadratic polynomial in $Y_0$, i.e., the coefficient of $Y_0^2$ is nonzero mod $p$. Furthermore, all such quadratic polynomials occur with equal probability.
By Lemma \ref{L: affinebinquad}, this quadratic splits into distinct factors over $\F_p$ with probability 
$\frac{1}{2}(1-\frac{1}{p})$. In this case, the point $((0:1),(\alpha:1))$, where $\alpha$ is one of the roots of the quadratic, is a
smooth $\F_p$-point, so by Hensel's lemma, the curve defined by $F=0$ has a $\Q_p$-point of the 
form $((0:1),(\widetilde{\alpha}:1))$ for some lift $\widetilde{\alpha} \in \Z_p$. If the quadratic is instead
irreducible, as happens with probability $\frac{1}{2}(1-\frac{1}{p})$, then 
there are no $\Q_p$-points with $p \nmid X_1$.
It remains to consider the case $F \equiv a_{20} X_1^2 (Y_0 - c Y_1)^2 \pmod{p}$, 
for some $0 \le c \le p-1$. Transforming $F$ by the substitution 
$Y_0 \leftarrow Y_0 + c Y_1$ we find
\begin{equation}
\label{sigmatau}
\sigma \begin{pmatrix} a_{00} & a_{01} & a_{02} \\ a_{10} & a_{11} \\ a_{20} \end{pmatrix}
= \frac{1}{2}\bigg(1-\frac{1}{p}\bigg) + \frac{1}{p^2} \sum_{c=0}^{p-1} \tau
 \begin{pmatrix} a_{00} & 2ca_{00} + a_{01} & c^2a_{00} + ca_{01} + a_{02}  
\\ a_{10} & 2ca_{10} + a_{11} \\ a_{20} \end{pmatrix}.
\end{equation}
We note in particular that the arguments of $\tau$ 
satisfy the conditions in the statement of the lemma.

If $F \in \calT$ then $F \equiv a_{20} X_1^2 Y_0^2 \pmod{p}$. For a solution with
$p \nmid X_1$ we need $p \mid Y_0$. This suggests making the substitution 
$Y_0 \leftarrow p Y_0$. Dividing through by $p$, and then swapping the $X$'s and $Y$'s
we find
\begin{equation}
\label{tausigma}
 \tau \begin{pmatrix} a_{00} & a_{01} & a_{02} \\ a_{10} & a_{11} \\ a_{20} \end{pmatrix}
=  \sigma \begin{pmatrix} pa_{00} & pa_{10} & pa_{20} \\ a_{01} & a_{11} \\ p^{-1} a_{02} \end{pmatrix}.
\end{equation}

Using~\eqref{sigmatau} and~\eqref{tausigma} to solve for $\sigma$ and $\tau$, we find that
$\sigma = \tau = 1/2$.
\end{proof}

\subsection{The cases without smooth $\F_p$-points}
\label{ss: nosmoothpts}

As remarked before, these are the factorization types of the $(2,2)$-form $F$ over $\F_p$ which do not immediately yield a smooth $\F_p$-point in the reduction of the corresponding curve.
\[ \begin{array}{cll}
\text{Case number} &
\text{Factorization type} &  
\text{Number of forms up to scaling by $\F_p^\times$} \\
\hline \\[-10pt]
1 & (1,1)\overline{(1,1)} & \quad n_1 := (p^3-p)(p^3 + p - 1)/2 \\
2 & (2,0)(0,2)            & \quad n_2 := p^2(p - 1)^2/4 \\
3 & (2,0)(0,1)^2 \text{ or } (0,2)(1,0)^2 &  \quad n_3 := p(p + 1)(p - 1) \\
4 & (1,1)^2        &  \quad n_4 := p(p + 1)(p - 1)  \\
5 & (1,0)^2(0,1)^2 &  \quad n_5 := (p + 1)^2
\end{array} \]

Let 
$n_0 = (p^9-1)/(p-1) - (n_1 + n_2 + n_3 + n_4 + n_5)$
be the number of forms lying in none of the $5$ cases, and 
let $\xi_i$ be the probability of solubility in case $i$.
Then the overall probability of solubility is 
\[ \rho = \frac{ n_0 + n_1 \xi_1 + n_2 \xi_2 + n_3 \xi_3 +  n_4 \xi_4
     + n_5 \xi_5 } { (p^9-1)/(p-1) } \]

In this section we compute $\xi_1, \ldots, \xi_5$ and hence
obtain the final answer stated in Theorem~\ref{thm:loc}.
In the context of computing $\xi_5$, it is helpful to make
the following definition for $\xi'_i$, and to compute the $\xi'_i$ alongside the $\xi_i$:
\begin{definition}
\label{def:nuprime}
For $1 \le i \le 4$, we let $\xi'_i$ be the probability of solubility
given the following conditions: we are in case $i$, the point $((0:1),(0:1))$
is a singular point on the reduction mod $p$,
and $v(a_{22}) \ge 2$. We write $\xi'_{1j}$ for $1 \leq j \leq 3$ for the same
probability in cases 1(i), 1(ii), and 1(iii), as defined below.
We define $\xi'_5$ in the same way, 
except that we require that the singular point $((0:1),(0:1))$
is {\em not} the point where the two lines meet.
\end{definition}

We compute the values of $\xi_i$ and $\xi'_i$ in the next sections.

\subsubsection{Case 1} In this case the reduction of our curve mod $p$ is
geometrically the union of two rational curves. We subdivide into
the cases (i), (ii), (iii) as defined immediately before Lemma~\ref{lem:count1}, and note that
$n_1 = n_{11} + n_{12} + n_{13}$.
Writing $\xi_{11}, \xi_{12}, \xi_{13}$ for the probabilities of
solubility in cases 1(i), 1(ii) and 1(iii), respectively, we have
\[ \xi_1 = (n_{11} \xi_{11} + n_{12} \xi_{12} + n_{13} \xi_{13})/n_1. \]

\subsubsection*{Case 1(i)}
In this case, the two components meet at a 
pair of points defined over $\F_p$. We begin by computing $\xi_{11}$. 
As in the proof of Lemma~\ref{lem:count1}(i) 
we may assume that
$\overline{F} = f(X_0Y_1,X_1Y_0)$ for some irreducible binary
quadratic form $f$. 
The only $\F_p$-points on the reduction are the singular points 
$((0:1),(0:1))$ and $((1:0),(1:0))$. We must decide if they
lift to $\Q_p$-points. Let $\alpha$ be the probability that the
singular point $((0:1),(0:1))$ lifts.

Since $((0:1),(0:1))$ is a singular point on the curve defined by 
$\overline{F}$, 
we deduce that all the valuations 
$v(a_{12}), v(a_{21}), v(a_{22})$ are $\geq 1$. If $v(a_{22}) = 1$, then the singular point $((0:1),(0:1))$ does not lift. Otherwise (with probability $1/p$), we have $v(a_{22}) \ge 2$. Then the valuations of the coefficients of $F$ satisfy
\begin{center}
\begin{tabular}{ccc}
$\geq 1$ & $\geq 1$ & $  =  0$\\
$\geq 1$ & $\geq 0$ & $\geq 1$\\
$  =  0$ & $\geq 1$ & $\geq 2$
\end{tabular}
\end{center}
where the equalities follow from $f$ being irreducible.
Making the substitutions $X_0 \leftarrow pX_0$,
$Y_0 \leftarrow pY_0$ and dividing through by $p^2$, we obtain a $(2,2)$-form $G(X_0,X_1,Y_0,Y_1)$
whose coefficients $b_{ij}$ have valuations satisfying
\begin{center}
\begin{tabular}{ccc}
$\geq 3$ & $\geq 2$ & $  =  0$\\
$\geq 2$ & $\geq 0$ & $\geq 0$\\
$  =  0$ & $\geq 0$ & $\geq 0 \rlap{.}$
\end{tabular}
\end{center}

We now investigate whether 
$\overline{G} \in \F_p[X_0,X_1;Y_0,Y_1]$ is absolutely
irreducible.
Define a ternary quadratic form $Q(X_0,Y_0,Z_0)$ by
$\overline{G}(X_0,1,Y_0,1) = Q(X_0,Y_0,1)$, so that the zero-set of $Q$
in $\Aff^2_{X_0,Y_0} \subset \PP^2$ coincides with the zero-set of
$\overline{G}$ in
$\Aff^1_{X_0} \times \Aff^1_{Y_0} \subset \PP^1 \times \PP^1$. Then
the curve defined by $Q$ (and thus, the curve defined by
$\overline{G}$) is geometrically irreducible if and only if the
discriminant of $Q$ is nonzero, equivalently
\[ b_{22} \disc(f) - f(b_{21},-b_{12}) \not\equiv 0 \pmod{p}. \]
(This argument still works in characteristic $2$ provided that the formula for the discriminant of a ternary quadratic form is scaled by appropriate powers of $2$.)

If the curve defined by $\overline{G}$ is geometrically irreducible, then the argument in Lemma \ref{lem:irred} shows that it has a smooth $\F_p$-point.
Otherwise (with probability $1/p$), the reduction mod $p$ is geometrically the union of
two rational curves meeting at $((1:0),(1:0))$ and an $\F_p$-point of the form
$((\lambda:1),(\mu:1))$. The two rational curves are not defined over $\F_p$,
since the binary quadratic form $f$ is irreducible. 
We make the substitutions 
$X_0 \leftarrow X_0 + \lambda  X_1$ and $Y_0 \leftarrow Y_0 + \mu  Y_1$
to move the second point of intersection to $((0:1),(0:1))$, 
and start over again
considering whether this singular point lifts.
The probability that it lifts is again $\alpha$.
We thus obtain the recursive formula
\[\alpha = \frac{1}{p}\left(\left(1 - \frac{1}{p}\right) + \frac{1}{p}
\,  \alpha \right),\]
and so $\alpha = 1/(p+1)$.

We are interested in the probability
that at least one of the singular points lifts. Since these
events depend on different coefficients of the $(2,2)$-form they
are independent. Therefore
\[  \xi_{11} = 1 - \left( 1- \frac{1}{p+1} \right)^2 = \frac{2p + 1}{(p+1)^2}. \]
A small modification of this argument (as required by
Definition~\ref{def:nuprime}) gives
\[  \xi'_{11} = 1 - \left( 1- \frac{1}{p+1} \right)\left( 1- \frac{p}{p+1} \right) = \frac{p^2 + p + 1}{(p+1)^2}. \]

\subsubsection*{Case 1(ii)} 
The two components meet at 
a pair of conjugate points defined over $\F_{p^2}$. 
There are no $\F_p$-points on the reduction.
Therefore $\xi_{12}=0$ and $\xi'_{12}$ is not defined.

\subsubsection*{Case 1(iii)}  
The two components meet at 
a single point defined over $\F_p$.
As in the proof of Lemma~\ref{lem:count1}(iii) 
we may assume that
$\overline{F} = f(X_0Y_0,X_0Y_1 + X_1Y_0)$ for some irreducible binary
quadratic form $f$. 
The only $\F_p$-point on the reduction is the singular point 
$((0:1),(0:1))$. If this lifts to a $\Q_p$-point then we must have
$v(a_{22}) \ge 2$. Since this is exactly the condition in 
Definition~\ref{def:nuprime} it follows that $\xi_{13} = (1/p) \xi'_{13}$.

We show in Section~\ref{sec:quartics}, using results from
\cite{BCF-binaryquartic}, that 
\[ \xi_{13} = \frac{2 p^{10} + 3 p^9 - p^5 
          + 2 p^4 - 2 p^2 - 3 p - 1}{2 (p + 1)^2(p^9 - 1)}. \]
As noted in the last paragraph, we have $\xi'_{13} = p \xi_{13}$.

\subsubsection{Case 2} This is the case $(2,0)(0,2)$.
There are no $\F_p$-points on the reduction.
Therefore $\xi_{2} = 0$ and $\xi'_{2}$ is not defined.

\subsubsection{Case 3} This is the case
$(2,0)(0,1)^2$ or $(0,2)(1,0)^2$.
We may assume without loss of generality that $\overline{F} = f(X_0,X_1) Y_0^2$
for some irreducible binary quadratic form $f$. The coefficients
satisfy
\begin{center}
\begin{tabular}{ccc}
$= 0$ & $\geq 1$ & $\geq 1$\\
$\geq 0$ & $\geq 1$ & $\geq 1$\\
$= 0$ & $\geq 1$ & $\geq 1$
\end{tabular}
\end{center}
where the equalities follow from $f$ being irreducible.
Making the substitution $Y_0 \leftarrow p Y_0$ and dividing
through by $p$ gives
\begin{center}
\begin{tabular}{ccc}
$= 1$ & $\geq 1$ & $\geq 0$\\
$\geq 1$ & $\geq 1$ & $\geq 0$\\
$= 1$ & $\geq 1$ & $\geq 0$
\end{tabular}
\end{center}
The reduction mod $p$ is now $g(X_0,X_1) Y_1^2$
for some binary quadratic form $g$. If $g$ is irreducible,
splits, or has repeated roots, then the probability of solubility
is $0$, $1$, or $1/2$, respectively. In the last of these cases, we are using
Lemma~\ref{L: solubility1/2}: more specifically, 
we assume the double root is at $(X_0:X_1) = (0:1)$, make
the substitution $X_0 \leftarrow p X_0$, divide through by $p$,
and then apply the lemma. 
Note that the lemma applies as $v(a_{20})=0$, and there are
no solutions with $p \mid X_1 Y_1$ in view of the substitutions
we made to reach this situation.

If $g$ is identically zero, then we divide through by $p$, to obtain
a $(2,2)$-form $H$ satisfying the {\em line condition}, by which we mean
that $H(X_0,X_1;1,0) \pmod{p}$
is an irreducible binary quadratic form. Writing $\delta_{\rm line}$
for the probability of solubility in this case, we have
\begin{equation}
\label{eqn:nu3a}
\xi_3 = \frac{1}{p^3} \left( \frac{p^3 - p}{2} \cdot 1 + \frac{p(p-1)^2}{2} \cdot 0 +
(p^2 - 1) \cdot \frac{1}{2}
  + \delta_{\rm line} \right) 
\end{equation}
and 
\begin{equation}
\label{eqn:nu3prime}
\xi'_3 = \frac{1}{p^2} \left( p(p-1) \cdot 1 + (p - 1) \cdot \frac{1}{2}
  + \delta_{\rm line} \right).
\end{equation}

\begin{lemma}
There are (up to scaling by $\F_p^\times$) exactly $p^7(p-1)/2$ forms over $\F_p$
satisfying the line condition. The numbers of these in Cases 1 to 3 are
\begin{align*}
 r_{11} &= p^3(p+1)(p-1)^2/4 \\
 r_{12} &= p^2(p+1)(p-1)^2(p-2)/4 \\
 r_{13} &= p^2(p+1)(p-1)^2/2 \\
 r_2 &= p^2(p - 1)^2/4 \\
 r_3 &= p^2(p - 1)/2 
\end{align*} 
There are none in Cases 4 and 5.
\end{lemma}

\begin{proof}
It is easy to check that forms in Cases 1(ii) and 2 always satisfy
the line condition, and those in Cases 4 and 5 never satisfy the
line condition. By double counting pairs consisting of $(2,2)$-forms
and $(0,1)$-forms (both up to scalars) that meet in 
a pair of conjugate points over $\F_{p^2}$, we find that
$(p+1)r_{11} = (p-1) n_{11}$ and $(p+1)r_{13} = p n_{13}$.
In Case~3 we must count the forms $f(X_0,X_1) g(Y_0,Y_1)^2$
where $f$ is an irreducible binary quadratic form, and $g$ is a
linear form with $g(1,0) \not=0$. We find that $r_3$ is the product
of the $(p^2-p)/2$ choices for $f$ and the $p$ choices for $g$.
\end{proof}

Let $r_0 = p^7(p-1)/2 - (r_{11} + r_{12} + r_{13} + r_2 + r_3)$ be
the number of forms satisfying the line condition not in Cases~1 to~3. Then 
\begin{equation}
\label{eqn:nu3b}
\delta_{\rm line} = \frac{r_0 + r_{11} \xi_{11} + r_{12} \xi_{12} 
 + r_{13} \xi_{13} + r_2 \xi_2 + r_3 \xi_3}{p^7(p-1)/2} 
\end{equation}
Using the values of $\xi_{11}$, $\xi_{12}$, $\xi_{13}$ and $\xi_2$
already computed, we can now solve~\eqref{eqn:nu3a}
and~\eqref{eqn:nu3b} for $\xi_3$ and $\delta_{\rm line}$. We then 
use~(\ref{eqn:nu3prime}) to compute $\xi'_3$. We find that
\[ \xi_3 = \frac{p^{10} + 2 p^9 + p^6 - 2 p^5 
                 + 2 p^3 + p^2 - 3 p - 2}{2 (p+1)(p^9 - 1) }  \]
and 
\[ \xi'_3 = \frac{2 p^{10} + p^9 + p^7 - 2 p^6 + 2 p^4
              + p^3 - 2 p^2 - 2 p - 1}{2 (p+1) (p^9-1) }. \]

\subsubsection{Case 4} This is the case $(1,1)^2$. By a 
change of coordinates we may assume 
\[ F \equiv (X_0 Y_1 - X_1 Y_0)^2 \pmod{p}.\]

We show in Section~\ref{sec:quartics}, using results from
\cite{BCF-binaryquartic}, that 
\[ \xi_4 = \frac{5 p^{10} + 8 p^9 + p^8 - p^7 + 2 p^6 
              - 3 p^5 + 4 p^3 - 10 p - 6}{8 (p+1) (p^9-1)}, \]
and
\[ \xi'_4 = \frac{4 p^{10} + 3 p^9 - p^7 + 2 p^6 - 2 p^5 
            + 2 p^3 - p^2 - 5 p - 2}{4 (p+1) (p^9-1)}. \]

\subsubsection{Case 5}
This is the case $(1,0)^2(0,1)^2$. By a 
change of coordinates we may assume 
\[F \equiv X_0^2 Y_0^2 \pmod{p}.\]
The coefficients of $F$ have valuations satisfying
\begin{center}
\begin{tabular}{ccc}
$=0$ & $\geq 1$ & $\geq 1$\\
$\geq 1$ & $\geq 1$ & $\geq 1$\\
$\geq 1$ & $\geq 1$ & $\geq 1\rlap{.}$
\end{tabular}
\end{center}
Let $Q$ and $Q'$ be the binary quadratic forms over $\F_p$ determined
by the last row and column, i.e.,
\begin{align*}
Q(Y_0,Y_1) &= \tfrac{1}{p} F(0,1;Y_0,Y_1) \pmod{p} \\
Q'(X_0,X_1) &= \tfrac{1}{p} F(X_0,X_1;0,1) \pmod{p} 
\end{align*}
Note that these forms have the same last coefficient $c \in \F_p$ as they share one
entry in the coefficient matrix corresponding to $X_1^2Y_1^2$.
Writing $\xi_{51}$ and $\xi_{52}$ for the probabilities
of solubility in the cases $c \not= 0$ and $c=0$, respectively, we have
\[ \xi_5 = (1-1/p) \xi_{51} + (1/p) \xi_{52}. \]

\begin{lemma}
\label{lem:nu51}
 We have $\xi_{51} = 3/4$.
\end{lemma}
\begin{proof}
The coefficients satisfy
\begin{center}
\begin{tabular}{ccc}
$=0$ & $\geq 1$ & $\geq 1$\\
$\geq 1$ & $\geq 1$ & $\geq 1$\\
$\geq 1$ & $\geq 1$ & $= 1 \rlap{.}$ 
\end{tabular}
\end{center}
The reduction mod $p$ is the union of two double lines, meeting
at a single point. Any $\Q_p$-point has $p \mid X_0$ or $p \mid Y_0$, but
not both since $v(a_{22})=1$. In other words, any $\Q_p$-point must
reduce to lie on exactly one of the lines.

To investigate whether there are solutions with $p \mid X_0$ we make 
the substitution $X_0 \leftarrow pX_0$ and divide by $p$ to get
\begin{center}
\begin{tabular}{ccc}
$=1$ & $\geq 2$ & $\geq 2$\\
$\geq 1$ & $\geq 1$ & $\geq 1$\\
$\geq 0$ & $\geq 0$ & $= 0\rlap{.}$
\end{tabular}
\end{center}
We then apply Lemma~\ref{L: solubility1/2} (with $Y_0 \leftrightarrow Y_1$).
The probability of a solution with $p \mid X_0$ and the probability of a solution with $p \mid Y_0$ are each
$1/2$. The lemma also implies that these two events are independent of
each other, so the probability of insolubility of these polynomials is
$1/4$. Hence, the probability of solubility is $3/4$.
\end{proof}
 
\begin{definition}
Let $\delta_1$ and $\delta_2$ be the probabilities of 
solubility in the cases 
\begin{center}
\begin{tabular}{ccc}
$\geq 1$ & $\geq 1$ & $\geq 0$\\
$\geq 1$ & $\geq 0$ & $\geq 0$\\
$   = 0$ & $\geq 0$ & $\geq 0$
\end{tabular}
\qquad 
\begin{tabular}{ccc}
$\geq 1$ & $\geq 1$ & $   = 0$\\
$\geq 1$ & $\geq 0$ & $\geq 0$\\
$   = 0$ & $\geq 0$ & $\geq 0$
\end{tabular}
\end{center}
(The subscript is the number of equalities in the matrix.)
Let $\delta^*_1$ and $\delta^*_2$ be the probabilities when we change
the top left $\geq 1$ to $=1$. 
Let $\varepsilon_1$ and $\varepsilon_2$ be the probabilities when we change
the top left $\geq 1$ to $\geq 2$. 
\end{definition}

Clearly we have
\begin{equation}
\begin{aligned}
\delta_1 = ( 1- 1/p) \delta^*_1 + (1/p) \varepsilon_1 \\
\delta_2 = ( 1- 1/p) \delta^*_2 + (1/p) \varepsilon_2 
\end{aligned}
\end{equation}

\begin{lemma} 
We have
\[ \xi_{52} = \left( 1 - \frac{1}{p^2} \right)
  + \frac{1}{p^2} \left(  \frac{1}{p}  \left( 1- \frac{1}{p} \right)^2 \,
\delta^*_2 + 2 \, \frac{1}{p} \left( 1 - \frac{1}{p} \right) \delta^*_1
    + \left( \frac{1}{p} \right)^2 \varepsilon_1 \right) \] 
\end{lemma}
\begin{proof}
If at least one of the forms $Q$ and $Q'$ has distinct roots in $\F_p$
then the $(2,2)$-form is soluble over $\Q_p$. This happens with
probability $1- 1/p^2$. Otherwise (with probability $1/p^2$) the coefficients
satisfy
\begin{center}
\begin{tabular}{ccc}
$   = 0$ & $\geq 1$ & $\geq 1$\\
$\geq 1$ & $\geq 1$ & $\geq 2$\\
$\geq 1$ & $\geq 2$ & $\geq 2$
\end{tabular}
\end{center}
(The bottom right $\geq 2$ comes from the assumption that $c = 0$, and the two adjacent $\geq 2$ entries arise from assuming that neither $Q$ nor $Q'$ has distinct roots in $\F_p$.)
We split into 3 cases:
\begin{enumerate}[topsep=0.5\baselineskip,itemsep=0.5\baselineskip]
\item Suppose $v(a_{02}) = v(a_{20})=1$. If $v(a_{22})=2$ then 
the $(2,2)$-form is insoluble over $\Q_p$. 
Otherwise, we find by substituting $X_0 \leftarrow p X_0$,
$Y_0 \leftarrow p Y_0$, and dividing through by $p^3$, that
the probability of solubility is $\delta^*_2$.
\item Suppose $v(a_{02}) = 1$ and $v(a_{20}) \ge 2$. We find by
substituting $X_0 \leftarrow p X_0$, and dividing through by $p^2$, 
that the probability of solubility is $\delta^*_1$. 
The case where $v(a_{02}) \ge 2$ and $v(a_{20}) = 1$ works in 
exactly the same way  via the substitution $Y_0 \leftarrow p Y_0$. 
\item Suppose $v(a_{02}) \ge 2$ and $v(a_{20}) \ge 2$.
Via either of the substitutions in (ii),
the probability of solubility is $\varepsilon_1$.
\end{enumerate}
Combining these gives the desired expression for $\xi_{52}$.
\end{proof}

\begin{lemma} 
\label{lem:nu5prime}
We have
\[ \xi'_{5} = \left( 1 - \frac{1}{p} \right)
  + \frac{1}{p} \left( 1- \frac{1}{p} \right) \frac{3}{4}
  + \left( \frac{1}{p} \right)^2 \left( 1- \frac{1}{p} \right)
 + \left( \frac{1}{p} \right)^3 \delta_1. \]
\end{lemma}
\begin{proof}
According to Definition~\ref{def:nuprime},
we may suppose the coefficients of $F$ satisfy
\begin{center}
\begin{tabular}{ccc}
$   = 0$ & $\geq 1$ & $\geq 2$\\
$\geq 1$ & $\geq 1$ & $\geq 1$\\
$\geq 1$ & $\geq 1$ & $\geq 1$
\end{tabular}
\end{center}
If $v(a_{12}) = 1$, then $Q'$ has distinct roots
in $\F_p$ and so $F$ is soluble over $\Q_p$. Otherwise (with
probability $1/p$), we have $v(a_{12}) \ge 2$. 

If $v(a_{22}) = 1$, then
by an argument 
similar to Lemma~\ref{lem:nu51},
the probability of
solubility is $3/4$. 
(The solutions with $p \mid X_0$ are analysed in exactly the
same way as before, whereas to analyse those with $p \mid Y_0$ we
substitute $Y_0 \leftarrow p Y_0$ and then $X_1 \leftarrow pX_1$.)

Otherwise (with probability $1/p$), we have
$v(a_{22}) \ge 2$. If $v(a_{21}) = 1$, then $Q$ 
has distinct roots in $\F_p$ and so $F$ is soluble over
$\Q_p$. Otherwise (with probability $1/p$), we have $v(a_{21}) \ge 2$.
We find by substituting $Y_0 \leftarrow p Y_0$ and then dividing 
through by $p^2$ that the 
probability of solubility is $\delta_1$.
\end{proof}

\begin{lemma}
There are $p^5$ possibilities
for $\overline{F}$ (up to scaling by $\F_p^\times$) 
satisfying the conditions in the definition of 
$\delta_1$. The number of these in Cases 1 to 5 are
\begin{align*}
s_{11} &= p^3 (p-1)/2, \\
s_{12} &= 0,           \\
s_{13} &= p (p-1)^2/2, \\
s_2 &= 0,             \\
s_3 &= p (p-1)/2,     \\
s_4 &= p (p-1)        \\
s_5 &= p.            
\end{align*}
In Cases 1(i), 1(iii), and 4, these forms also satisfy the conditions
in the definition of $\delta_2$. In Cases 3 and 5, they do not.
\end{lemma}
\begin{proof}
The conditions in the definition of $\delta_1$ are that 
$\overline{F}=0$ is singular at $((1:0),(1:0))$ but  
does not contain the line $Y_1=0$.
We have $(p+1)^2 s_{11} = 2 n_{11}$ and $(p+1)^2 s_{13} = n_{13}$.
In Cases~1(ii) and~2, there are no $\F_p$-points so 
$s_{12} = s_{2} = 0$. In Cases~3, 4, and 5, we count the forms
$X_1^2 f(Y_0,Y_1)$, $(\alpha X_1 Y_0 + \beta X_0 Y_1 + \gamma X_1 Y_1)^2$,
and $X_1^2 g(Y_0,Y_1)^2$ where $f$ is an irreducible binary 
quadratic form, $\alpha, \beta, \gamma \in \F_p^\times$ with
$\alpha \beta \not= 0$, and $g$ is a linear form with $g(1,0) \not= 0$.
Finally, it is only in Cases 3 and 5 that the reduction mod $p$
contains the line $X_1 = 0$.
\end{proof}

For the final computation of $\xi_5$, let 
$s_0 = p^5 - (s_{11} + s_{13} + s_{3} + s_{4} + s_{5})$
and $t_0 = p^4(p-1) - (s_{11} + s_{13} + s_{4})$ so that
\begin{align*}
\delta_1 &= \frac{s_0 + s_{11} \xi_{11} + s_{13} \xi_{13} 
                  + s_3 \xi_3 + s_4 \xi_4 + s_5 \xi_5}
  {p^5}  \\
\delta_2 &= \frac{t_0 + s_{11} \xi_{11} + s_{13} \xi_{13} 
                  + s_4 \xi_4}
  {p^4(p-1)}  .
\end{align*}
Replacing each $\xi$ by $\xi'$ (see Definition~\ref{def:nuprime})  
we have
\begin{align*}
\varepsilon_1 &= \frac{s_0 + s_{11} \xi'_{11} + s_{13} \xi'_{13} 
                  + s_3 \xi'_3 + s_4 \xi'_4 + s_5 \xi'_5}
  {p^5}  \\
\varepsilon_2 &= \frac{t_0 + s_{11} \xi'_{11} + s_{13} \xi'_{13} 
                  + s_4 \xi'_4}
  {p^4(p-1)}  .
\end{align*}
Putting together all the equations derived in this section, together
with the previously computed $\xi$'s, we are now
able to solve for $\xi_5$. We find $\xi_5 = f(p)/g(p)$ where
\begin{align*}
f(p) &= 6 p^{18} + 8 p^{17} + 2 p^{16} - 8 p^{15} + 16 p^{14} - 12 p^{13} - 4 p^{12} + 3 p^{11} + 9 p^{10} - 35 p^9 \\ &~\qquad 
   + 8 p^8 - 11 p^7 + 3 p^6 - p^5 + 8 p^4 - 6 p^3 - 4 p^2 + 10 p + 8, \\
g(p) &= 8 (p+1) (p^9-1) (p^8-1).
\end{align*}

\section{Relation to binary quartics}
\label{sec:quartics}

We compute some of the probabilities required in Section~\ref{sec:mainpf}
by reducing them to probabilities already computed in
\cite{BCF-binaryquartic}. The basic idea is that a $(2,2)$-form
determines a binary quartic form, by writing the $(2,2)$-form as a binary
quadratic form in $Y_0,Y_1$ and taking the discriminant.
However, since we also want results in the case $p=2$, we will in 
fact work with generalised binary quartics, defined as follows.

\begin{definition}
A {\em generalised binary quartic} $(G_2,G_4)$ is a pair of binary forms
of degrees 2 and 4. A generalised binary quartic $(G_2,G_4)$ is {\em soluble} 
over a field $K$ if for some $X_0, X_1, Z \in K$ with $(X_0,X_1) \not= (0,0)$ 
we have $Z^2 + G_2(X_0,X_1) Z = G_4(X_0,X_1)$.
\end{definition}

We write $\Z_p[X_0,X_1] = \oplus_d \Z_p[X_0,X_1]_d$ 
and $\Z_p[X_0,X_1;Y_0,Y_1] = \oplus_{d,e} \Z_p[X_0,X_1;Y_0,Y_1]_{de}$ 
for the gradings of these rings by degree $d$ and by bidegree $(d,e)$, respectively.

\begin{lemma}
\label{lem:BCF1}
Let $\ell,a \in \F_p$ such that $Z^2 + \ell Z - a$ is irreducible over $\F_p$, and let
\begin{align*}
\calS &:= \{ (G_2,G_4) \in \Z_p[X_0,X_1]_2 \times \Z_p[X_0,X_1]_4 : 
G_2 \equiv \ell X_0^2 \!\!\!\pmod{p} \text{ and } G_4 \equiv a X_0^4 \!\!\!\pmod{p} \}, \\
\calT &:= \{ (G_2,G_4) \in \Z_p[X_0,X_1]_2 \times \Z_p[X_0,X_1]_4 : 
G_2 \equiv 0 \!\!\!\pmod{p} \text{ and } G_4 \equiv 0 \!\!\!\pmod{p} \}, 
\end{align*}
and $\calT^* := \{ (G_2,G_4) \in \calT : G_4(0,1) \not\equiv 0 \pmod{p^2} \}$.
Then the proportions of generalised binary quartics in $\calS$, 
$\calT$ and $\calT^*$
that are soluble over $\Q_p$ are, respectively,
\[ \sigma = \frac{2 p^{10} + 3 p^9 - p^5 
          + 2 p^4 - 2 p^2 - 3 p - 1}{2 (p + 1)^2(p^9 - 1)}, \]
\[ \tau = \frac{5 p^{10} + 8 p^9 + p^8 - p^7 + 2 p^6 
              - 3 p^5 + 4 p^3 - 10 p - 6}{8 (p+1) (p^9-1)}, \]
and
\[ \tau^* = \frac{5 p^{10} + 5 p^9 - p^7 + 3 p^6
                  - 4 p^5 + 4 p^3 - 8 p - 4}{8 (p+1)(p^9-1)}. \]

\end{lemma}
\begin{proof}
These probabilities were computed in \cite{BCF-binaryquartic}.
The probability $\sigma$ was computed following Corollary 18, where it
was denoted $\lambda$. The probabilities $\tau$ and $\tau^*$ were computed in Sections 2.3.6 and
2.3.5, where they were denoted $\sigma_4$ and $\sigma'_4$.
\end{proof}

We define a map
\begin{align*}
\Phi : \qquad \Z_p[X_0,X_1;Y_0,Y_1]_{22} & \to \Z_p[X_0,X_1]_2 \times \Z_p[X_0,X_1]_4 \\
F_0 Y_0^2 + F_1 Y_0 Y_1 + F_2 Y_1^2 & \mapsto (F_1,-F_0 F_2). 
\end{align*}
It is easy to check that a $(2,2)$-form $F$ is soluble over $\Q_p$
if and only if the generalised binary quartic $\Phi(F)$ is soluble over
$\Q_p$. We write $\Phi_p$ for the corresponding map on
forms with coefficients in $\F_p$. 
The cases (i) and (ii) in the following lemma relate to Cases~1(iii)
and~4 in Section~\ref{ss: nosmoothpts}.

\begin{lemma}
\label{lem:measpres}
Let $\overline{F} \in \F_p[X_0,X_1;Y_0,Y_1]_{22}$ take one of the following forms:
\begin{enumerate}
\item $\overline{F} = f(X_0 Y_0, X_0 Y_1 + X_1 Y_0)$ where $f$ is an irreducible 
binary quadratic form,
\item $\overline{F} = (X_0 Y_1 - X_1 Y_0)^2$.
\end{enumerate}
Then $\Phi$ restricts to a measure-preserving map 
\begin{align*} \{ F \in \Z_p[X_0,X_1;Y_0,Y_1]_{22} & : 
  F \equiv \overline{F} \!\!\!\pmod{p} \} \\
& \to \{ G \in \Z_p[X_0,X_1]_2 \times \Z_p[X_0,X_1]_4 
              : G \equiv \Phi_p(\overline{F}) \!\!\!\pmod{p} \}. 
\end{align*}
\end{lemma}

\begin{proof}
The proof comes down to showing that the derivative of $\Phi_p$ 
at $\overline{F}$ is a surjective linear map $\F_p^9 \to \F_p^8$.
In cases (i) and (ii), this linear map is given by
\begin{align*} 
F_0 Y_0^2 + F_1 Y_0 Y_1 + F_2 Y_1^2 & \mapsto (F_1, - f(0,1) X_0^2 F_0 - f(X_0,X_1) F_2), \textrm{ and}\\
F_0 Y_0^2 + F_1 Y_0 Y_1 + F_2 Y_1^2 & \mapsto (F_1, - X_0^2 F_0 - X_1^2 F_2),
\end{align*} 
respectively, which are both surjective.
\end{proof}

\subsubsection*{Case 1(iii)}  
We use these lemmas to compute $\xi_{13}$, that is, the probability
of solubility where $F$ mod $p$ is of the form indicated 
in Lemma~\ref{lem:measpres}(i). Let $f$ have coefficients $a,b,c$. 
Then $\Phi_p(\overline{F})$ is the generalised binary quartic with
equation 
\[ (Z + c X_0 X_1)^2 + b X_0^2 (Z + c X_0 X_1) + a c X_0^4 = 0. \]
From this we see that $\xi_{13} = \sigma$ as defined in Lemma~\ref{lem:BCF1}.

\subsubsection*{Case 4}  
We use these lemmas to compute $\xi_{4}$, that is, the probability
of solubility where $F$ mod $p$ is of form indicated 
in Lemma~\ref{lem:measpres}(ii). 
Since $\Phi_p(\overline{F})$ is identically zero,
we see that $\xi_{4} = \tau$ as defined in Lemma~\ref{lem:BCF1}.
To compute $\xi'_{4}$ we must consider $(2,2)$-forms $F$ that
additionally satisfy $v(a_{22}) \ge 2$.
Under the measure preserving map in Lemma~\ref{lem:measpres}(ii)
these are mapped to $\calT \setminus \calT^*$. Therefore
$\tau = (1 - 1/p)\tau^*  + (1/p)\xi_4'$ and so
\[ \xi'_4 = p \tau - (p-1) \tau^* = \frac{4 p^{10} + 3 p^9 - p^7 + 2 p^6 - 2 p^5 
            + 2 p^3 - p^2 - 5 p - 2}{4 (p+1) (p^9-1)}.\]

\begin{remark} The same approach could be used to compute $\xi_{11}$, and 
indeed our answer agrees with \cite{BCF-binaryquartic}*{Lemma 15}.
\end{remark}

\section{Connections to the Hasse principle}
\label{sec:exp}

In Theorem~\ref{thm:global},
we determined that the proportion of $(2,2)$-forms 
that are everywhere locally soluble is $c \approx 0.8739$.
As we explain further in \S\S \ref{sec:exp-res} and \ref{sec:marked},
a heuristic similar to \cite[Conjectures 6 and 7]{manjul-hassecubic}
predicts that the proportion of everywhere locally soluble $(2,2)$-forms 
that are globally soluble is $\frac{1}{4}$, i.e., in the notation of \S \ref{S: PoonenVoloch},
that $\lim_{H \to \infty} \frac{N(H)}{\Nloc(H)} = \frac{1}{4}$
and thus $\lim_{H \to \infty} \frac{N(H)}{\Ntot(H)} = \frac{1}{4}c$.

In this section we report on some experiments to test this
conjecture numerically. A similar study in the case of plane
cubics was made in \cite{FisherOberwolfach}. With the one exception noted
below, all computations were performed using
Magma \cite{Magma} and the data may be found at \cite{dataexpts}.

\subsection{Experiments and results}
\label{sec:exp-res}
For each $H \in \{10,30,100,300,1000\}$, we chose $1000$ 
$(2,2)$-forms 
(i.e., polynomials of the form~\eqref{eq:bideg22}),
with coefficients chosen
uniformly at random from $[-H,H] \cap \Z$.
The numbers of these
that were soluble or everywhere locally soluble (ELS) were
as follows:
\[
\renewcommand{\arraystretch}{1.2}
\begin{array}{rccc} 
& \text{initial range} & \text{improved range} & \\
\multicolumn{1}{c}{H} & \sol & \sol & \els \\  \hline
10   &  [753, 755]  &     753     & 885 \\
30   &  [640, 652]  &     642     & 885 \\
100  &  [536, 582]  &     549     & 875 \\
300  &  [378, 502]  &  [432, 433] & 867 \\
1000 &  [275, 464]  &  [357, 464] & 879 
\end{array}
\]
The second column gives our initial estimate for
the number of soluble $(2,2)$-forms out of the $1000$. The lower
bound was obtained by searching for rational solutions, with
the assistance of 4-descent 
in Magma. 
The upper bound was obtained by computing the 
Cassels-Tate pairing on the $2$-Selmer group of the Jacobian.
For the improved estimates in the third column we used a 
range of methods, described more fully below, that are conditional
on standard conjectures and sometimes were only practical 
for $H$ sufficiently small.

For the first 4 experiments, we also give the breakdown of these
totals by the rank of the Jacobian elliptic curve $E/\Q$. 
The annotations $+$ and $-$ in the case $H = 300$ indicate
that we should add or subtract one if the remaining form
whose solubility has not yet been decided turns out to be soluble.
\[
\renewcommand{\arraystretch}{1.2}
 \begin{array}{cc|cccccccc|c} 
& \rank E(\Q) & {\rm singular} & 0 & 1 & 2 & 3 & 4 & 5 & 6 & {\rm Total} \\ 
\hline

H = 10
& \els & 5 & 0 & 86 & 344 & 313 & 116 & 21 & 0 & 885 \\
& \sol & 5 & 0 & 16 & 285 & 310 & 116 & 21 & 0 & 753 \\[5pt]

H = 30
& \els & 0 & 0 & 122 & 310 & 291 & 129 & 29 & 4 & 885 \\
& \sol & 0 & 0 & 1 & 208 & 274 & 126 & 29 & 4 & 642 \\[5pt]

H = 100
& \els & 0 & 0 & 171 & 321 & 257 & 96 & 25 & 5 & 875 \\
& \sol & 0 & 0 & 0 & 205 & 221 & 93 & 25 & 5 & 549 \\[5pt]

H = 300 
& \els & 0 & 0 & \,\,\,230-\! & 373 
                & \,\,\,187+\! & 58 & 19 & 0 & 867 \\
& \sol & 0 & 0 & 0 & 210 & \,\,\,151+\! & 52 
                & 19 & 0 & \,\,\,432+\! 
\end{array} \]

Although we can see from our first table that the proportion of
everywhere locally soluble forms that are globally soluble is
decreasing with $H$, this hardly amounts to strong evidence that the
limit is $1/4$. However, the prediction of $1/4$ 
arises since, in the limit, it is expected that (i)
$50\%$ of the Jacobians have rank
1 and $50\%$ have rank 2 (by, e.g., the Minimalist Conjecture), and (ii) the proportions of forms 
that are soluble in these two cases are 
$0$ and $1/2$, respectively (as explained below).  Our second table therefore provides much stronger
evidence for the conjecture, and indeed we see that
the convergence in (ii) is happening much faster than that in (i).

\subsection{The marked point and heuristics}
\label{sec:marked}
As we saw in Section~\ref{sec:quartics}, if 
$F \in \Z[X_0,X_1;Y_0,Y_1]$ is a $(2,2)$-form,
then it determines a pair of binary quartics.
These binary quartics have the same discriminant, which is accordingly
called the discriminant of $F$.  We should expect a randomly chosen
$(2,2)$-form to have nonzero discriminant (and hence define a smooth
curve). This was true in all our experiments, except for $5$ cases
with $H = 10$, which were all in any case soluble. From now on we
assume that the discriminant is nonzero, and write $E$ for the
Jacobian of $C = \{F = 0\} \subset \PP^1 \times \PP^1$. Since the
discriminant is a degree $12$ polynomial in the coefficients of $F$,
the conductor and discriminant of $E$ each have size about
$H^{12}$.

There are two maps $C \to \PP^1$,
given by projection to each factor, and the difference of fibres
is a nonzero marked point $P_0 \in E(\Q)$. There is an explicit formula for $P_0$ 
(see \cite[Section 6.1.2]{BhargavaHo} or \cite[Lemma 2.1]{FisherRadicevic}) 
in terms of the coefficients of $F$. As might be predicted from
this formula, we found in our experiments that $P_0$ had canonical
height at most $\log(2H^2)$. The torsion subgroup of $E(\Q)$ was 
trivial in all but $6$ cases with $H = 10$, when it had order $2$.
In only one of these cases was $P_0$ a torsion point.

We should expect that for a randomly chosen $(2,2)$-form, 
the associated binary quartics should not have any rational roots
(i.e., linear factors). This was true in all but $69$, $9$ and $1$ 
of our examples with $H = 10, 30$ and $100$. We should also
expect that $P_0 \notin 2 E(\Q)$, and this was true in all but
$6$ cases when $H = 10$, and $2$ cases when $H = 30$.
If $P_0 \notin 2 E(\Q)$ and the associated binary quartics do not have
any rational roots, then if $\rank E(\Q) = 1$, the $(2,2)$-form is not soluble.
We thus expect that the $(2,2)$-form is not soluble in general if $\rank E(\Q) = 1$.

We now explain why half of the forms with rank $2$ Jacobian 
are expected to be soluble.
For an elliptic curve $E$ of rank $2$, we want to estimate the proportion
of elements in $\Sel^2(E/\Q)$ that are in the image of $E(\Q)/2E(\Q)$. The average
size of $\Sel^2(E/\Q)$ is $6$ in this family \cite{BH-counting}, but two of the $2$-Selmer group elements
correspond to $(2,2)$-forms for which one of the associated binary quartics
has a rational root, which should only happen 0\% of the time when we order
by height. Now the size of $E(\Q)/2E(\Q)$ is 4 for the 100\% of elliptic curves $E$
for which there is no $2$-torsion, but we also subtract the same $2$ elements corresponding
to the $(2,2)$-forms for which one of the associated binary quartics
has a rational root. We thus predict that $(4-2)/(6-2) = 1/2$ of the forms are soluble.

\begin{remark} Although one of the best methods for finding generators of large height on an
elliptic curve $E/\Q$ is to use Heegner points, this only works for curves of rank $1$.
Since our elliptic curves all come with a known point of infinite order, 
this method was of no use to us. We instead relied almost exclusively on descent methods.
\end{remark}

\subsection{The initial estimates: computing ranks of elliptic curves}
The curve defined by a $(2,2)$-form is isomorphic to the curve
defined by either of the associated binary quartics. Our interest
is therefore in deciding the solubility of the genus one curves
associated to binary quartics.

Let $C/\Q$ be a genus one curve defined by a binary quartic, 
and let $E/\Q$ be its Jacobian. If $C$ 
is everywhere locally soluble, then it defines a class $[C]$ 
in the $2$-Selmer group $\Sel^2(E/\Q)$. Moreover $C(\Q) \not=
\emptyset$ if and only if $[C] \in \im(\delta)$ 
where $\delta$ in the connecting map in the Kummer exact sequence
\[  0 \ra E(\Q)/2E(\Q) \stackrel{\delta}{\ra} \Sel^2(E/\Q) 
\ra \Sha(E/\Q)[2] \to 0. \]
Given a point $P \in E(\Q)$ the Magma function {\tt GenusOneModel(2,P)}
computes a binary quartic representing $\delta(P)$.
In conjunction
with the function {\tt IsEquivalent} for testing equivalence of binary
quartics, this gives a convenient way of reducing the problem of
deciding whether $C(\Q) \not= \emptyset$ to that of finding generators
for $E(\Q)$. 

An initial upper bound for the rank of $E(\Q)$ is
obtained by $2$-descent, that is, by 
computing the 2-Selmer group $\Sel^2(E/\Q)$. 
This upper bound can sometimes
be improved by computing the Cassels-Tate pairing.
Let $S_n$ be the image of the natural map $\Sel^{2^n}(E/\Q) \to \Sel^2(E/\Q)$.
If $\xi,\eta \in S_n$, say with $\xi' \mapsto \xi$ and $\eta' \mapsto \eta$,
then there is an alternating pairing 
\begin{equation}
\label{ctpn}
 \langle~,~\rangle_n :  
  S_n \times S_n \to \F_2 \, ; \quad (\xi,\eta) \mapsto \langle \xi', \eta
\rangle_{\rm CT} = \langle \xi, \eta' \rangle_{\rm CT} 
\end{equation}
whose kernel is $S_{n+1}$. We note the inclusions of $\F_2$-vector spaces
\[  \im(\delta) \subset \ldots \subset 
    S_3 \subset S_2 \subset S_1 = \Sel^2(E/\Q). \]
The function {\tt CasselsTatePairing} in Magma, written by S. Donnelly,
computes the pairings 
\begin{equation}
\label{ctp22}
 \langle~,~\rangle_{\rm CT} : \Sel^2(E/\Q) \times \Sel^2(E/\Q) \to \F_2, 
\end{equation}
and 
\begin{equation}
\label{ctp24}
 \langle~,~\rangle_{\rm CT} : \Sel^2(E/\Q) \times \Sel^4(E/\Q) \to \F_2, 
\end{equation}
taking as input either a pair of binary quartics, or a binary
quartic and a quadric intersection.
(A variant of his method for computing~\eqref{ctp22} is described
in \cite{FisherCTP}.) We may thus 
compute the pairing~\eqref{ctpn} for $n=1$ and $n=2$.

A lower bound for the rank of $E(\Q)$ may be obtained by searching
for points either directly on $E$, or better on one of its 
$4$-coverings as computed using {\tt FourDescent} in Magma.
In this way we obtained generators for a subgroup
$\Gamma \subset E(\Q)$ of known points.
In all cases where it is possible that $\rank \Gamma < \rank E(\Q)$ 
we searched up to height $10^{10}$ on the $4$-coverings.

Our initial (unconditional) estimate on the number of $(2,2)$-forms
that are soluble was obtained as follows. First, if the curve
$C$ defined by our $(2,2)$-form is not everywhere locally soluble, then it
is obviously not soluble. Otherwise $C$ (or more precisely one of
the associated binary quartics) determines a class 
$[C] \in \Sel^2(E/\Q)$. If $[C] \in \delta(\Gamma)$ then we 
know that $C(\Q) \not= \emptyset$, and indeed from the generators
for $\Gamma$ we may compute 
an explicit solution. Otherwise we look for $[D]  \in \Sel^2(E/\Q)$ with 
$\langle [C], [D] \rangle_{\rm CT} \not= 0$. If we succeed in finding
a binary quartic $D$ with these properties, then it 
is a witness to the fact that $C(\Q) = \emptyset$. 

\subsection{The improved estimates}
Both the improved estimates, and the second table (giving the
breakdown by rank of the Jacobian) are conditional on the following
two standard conjectures.
\begin{itemize}
\item Parity conjecture: This is the parity part of the Birch--Swinnerton-Dyer
conjecture, i.e., the Mordell-Weil rank of an elliptic curve
$E/\Q$ is even or odd according as its root number $w(E/\Q)$ is $+1$ or $-1$.
\item Generalised Riemann Hypothesis (GRH) : This is needed 
for the class number calculations for $2$-descent (but 
would be easy to remove for small $H$) and for the computation of 
analytic rank bounds.
\end{itemize}

For ease of exposition, we will assume (as is the case in all examples
of interest) that $E(\Q)$ has trivial torsion subgroup. We
write $r_{2^n} = \dim_{\F_2} S_n$ for the upper bound on the
rank of $E(\Q)$ obtained by $2^n$-descent. Thus we have
\[  \rank \Gamma \le \rank E(\Q) \le \ldots \le r_8 \le r_4 \le r_2. \]
It is a theorem, originally due to Monsky \cite{Monsky}, 
that $w(E/\Q) = (-1)^{r_2}$.

We improve our lower bounds on the number of forms that are soluble
by using the parity conjecture. 
Indeed, if $\rank \Gamma = r_{2^n} - 1$ for some $n$, then we may conclude
by the parity conjecture that $\im(\delta) = S_n$. We mainly used
this idea with $n=1$, when the conclusion is that 
binary quartics with Jacobian $E$ satisfy the Hasse principle,
but also used it with $n=2$ in four examples with $H = 1000$.

In the two examples where we used the parity conjecture in the
case $H = 30$, we were also able to find the missing generators
using {\tt EightDescent} in Magma \cite{StammingerThesis}. 
In particular on the curve $\{F = 0\} \subset \PP^1 \times \PP^1$ where
\begin{align*} 
 F &= 27 X_0^2 Y_0^2 - X_0 X_1 Y_0^2 + 17 X_1^2 Y_0^2 - 27 X_0^2 Y_0 Y_1 
    + 15 X_0 X_1 Y_0 Y_1 + 9 X_1^2 Y_0 Y_1 \\ & \qquad \qquad  - 25 X_0^2 Y_1^2 
    - 12 X_0 X_1 Y_1^2 - 13 X_1^2 Y_1^2, 
\end{align*}
we found the solution
\begin{align*}
X_0 &= 5998800628516423107297133082973646629266881508307007941326966876023, \\
X_1 &= 342294900150114936634770190317380320064921533929615189995360150770683, \\
Y_0 &= 246468494594162038245191010835877699291643209107952263886240062422805, \\
Y_1 &= -206172926328604047309514129427033995615708844556361901784128916991449.
\end{align*}
This maps to a point on the Jacobian of canonical height 
$\widehat{h} \approx 644.736$,
which is well beyond the range that could be found by $4$-descent.
Unfortunately it was not practical to run {\tt EightDescent} in 
the experiments with $H=100,300$ and $1000$, and so our improved lower bounds
in those cases remain conditional on the parity conjecture.

The main method we used to improve the upper bounds on the number 
of forms that are soluble was to compute the Cassels-Tate 
pairing~\eqref{ctp24}.
In the experiments with $H = 10,30, 100$ and $300$, we were left with $0,1,2$ 
and $5$ examples where $r_2 = r_4 = r_8 = 3$, yet (despite searching
on all $4$-coverings up to height bound $10^{10}$) we could only find
one generator. The elliptic curves in question are recorded in the
following table.
\begin{align*}
 y^2 &= x^3 - 385216 x - 118546643  & \Delta = 2.0 \\
 y^2 + x y &= x^3 - x^2 - 21940631 x - 10062163381  & \Delta = 2.6 \\
 y^2 + x y &= x^3 - x^2 - 130106786 x - 418444299752  & \Delta = 2.5 \\
 y^2 &= x^3 + x^2 - 674939767 x + 9768411280745  & \Delta = 3.5 \\
 y^2 + x y + y &= x^3 + 1365438724 x + 1088450102306  & \Delta = 3.0 \\
 y^2 + x y &= x^3 + x^2 + 13646956 x + 36868880351052  & \Delta = 3.6 \\
 y^2 + y &= x^3 + 463718380 x - 1653282652263  & \Delta = 3.9 \\
 y^2 + x y &= x^3 + x^2 - 6811523942 x + 180704627470189  & 
\end{align*}

In all but the last of these examples, we were able to prove 
that the rank is $1$
by using Sage \cite{SageMath} to compute an upper bound on the analytic rank.
The parameter $\Delta$ we used for this calculation (see \cite{Bober}) is 
recorded in the right hand column. 
In the last example we obtained
no rank bound better than $3$, despite taking $\Delta = 4.0$.

Unfortunately it was not practical to compute the pairing~\eqref{ctp24}  
in the experiment with $H = 1000$. So we are left with a large 
number of unresolved cases. Writing $t = \rank \Gamma$ 
for the number of generators
known, there were 90 cases with $(r_2,r_4,t) = ( 3, 3, 1 )$,
14 cases with $(r_2,r_4,t) = ( 4, 4, 2 )$, and one each with
$(r_2,r_4,t) = (4,4,1)$, $(5,3,1)$ and $(5, 5, 3)$.

\bibliography{hasse22}
\bibliographystyle{amsalpha}

\end{document}